\documentclass{amsart}
\usepackage{latexsym}
\usepackage{amsmath,amsfonts,epsfig, graphics, graphicx, amsthm, amssymb}
\input{xy}
\xyoption{all}

\newtheorem{theorem}{Theorem}[section]
\newtheorem{lemma}[theorem]{Lemma}
\newtheorem{corollary}[theorem]{Corollary}

\newtheorem*{GRH}{Generalized Riemann Hypothesis}

\newtheorem*{CorBT}{Corollary \ref{BTforprimes}}
\newtheorem*{LemFI}{Theorem \ref{FI:2.6}}
\newtheorem*{ThSh1}{Theorem \ref{Shp:T1}}
\newtheorem*{ThSh2}{Theorem \ref{Shp:T2}}

\pdfoutput=1

\begin{document}
\title{Prime Distribution and Siegel Zeroes}
\author[T. Wright]{Thomas Wright}
\address{Wofford College\\429 N. Church St.\\Spartanburg, SC 29302\\USA}
\maketitle

\begin{abstract}
Let $\chi$ be a Dirichlet character mod $D$ with $L(s,\chi)$ its associated $L$-function, and let $\psi(x,q,a)$ be, as usual, Chebyshev's prime-counting function for the primes of the arithmetic progression $a$ (mod $q$) with $(a,q)=1$.  For a fixed $R>7$, we prove that under the assumption of an exceptional character $\chi$ with $L(1,\chi)<(\log D)^{-R}$, there exists a range of $x$ for which the asymptotic
$$\psi(x,q,a)=\frac{\psi(x)}{\phi(q)}\left(1-\chi\left(\frac{aD}{(q,D)}\right)+o(1)\right)$$
holds for $q<x^{\frac{30}{59}-\varepsilon}$.  We also show slightly better bounds for $q$ if we take an average over a range of $q$, finding an Elliott-Halberstam-type result for $q\sim Q$ on the range $Q<x^{\frac{16}{31}-\varepsilon}$.
This improves on a Friedlander and Iwaniec 2003 result that requires $q<x^{\frac{233}{462}}$ and $R\geq 554,401^{554,401}$.

\end{abstract}

\section{Introduction}

In 1837, Peter Gustav Lejeune Dirichlet proved the prime number theorem in arithmetic progressions by introducing (in somewhat different notation) his eponymous $L$-function:
$$L(s,\chi)=\sum_{n=1}^\infty \frac{\chi(n)}{n^s}.$$
Here, $\chi$ is a Dirichlet character modulo an integer $q>2$.  We will assume that $\chi$ is non-principal, and hence the above sum is convergent for $Re(s)>0$.

Dirichlet's studies raised the question of when this function $L(s,\chi)$ equals zero.  In particular, the zero-free region around $s=1$ led to Dirichlet's theorem on the asymptotic density of prime numbers in arithmetic progressions, while larger zero-free regions would allow for better error terms for this theorem.  Indeed, one of the most famous conjectures in mathematics is the belief that all of these zeroes are, in fact, on the half-line:

\begin{GRH}
For a Dirichlet character $\chi$, let $L(s,\chi)=0$ for $s=\sigma+it$ with $\sigma>0$.  Then $\sigma=\frac 12$.
\end{GRH}

Of course, we are nowhere close to proving this.  In the case where the zero is real, the best effective bound comes from Landau's 1918 paper \cite{La}:
\begin{theorem}[Landau, 1918]
There exists an effectively computable positive constant $C$ such that for any $q$ and any character $\chi$ mod $q$, if $L(s,\chi)=0$ and $s$ is real, then
$$s<1-\frac{C}{q^\frac 12 \log^2 q}.$$
\end{theorem}
The best ineffective bound was proven by Siegel in 1935 \cite{Si}, who was able to improve the denominator in Landau's result to any $q^\varepsilon$ but at the cost of no longer being able to compute the associated constant:
\begin{theorem}[Siegel, 1935]
For any $\varepsilon>0$ there exists a positive constant $C(\varepsilon)$ such that if $L(s,\chi)=0$ and $s$ is real then
$$s<1-C(\varepsilon)q^{-\varepsilon}.$$
\end{theorem}
However, most zeroes are far closer to the half-line than these bounds indicate.  In fact, it is known (see \cite{Gr}, \cite{La}, \cite{Ti}) that for any $q$, every zero of $L(s,\chi)$ except at most one will obey a much smaller bound:
\begin{theorem}
There is an effectively computable positive constant $C$ such that $$\prod_{\chi\mbox{ }mod\mbox{ }q}L(s,\chi)=0$$
has at most one solution on the region $$\sigma\geq 1-\frac{C}{\log q(2+|t|)}.$$
If such a zero exists, $s$ must be real, and the character for which $L(s,\chi)=0$ must be a non-principal real character.
\end{theorem}


A zero of this type, if it is to exist, is called a \textit{Siegel zero} or an \textit{exceptional zero}, and the associated character is called an exceptional character. We note that the definition given here (or, indeed, in the literature in general) for a Siegel zero is not particularly rigorous, since this definition depends on the choice of the constant $C$.

\section{Siegel Zeroes}

While the existence of Siegel zeroes would unfortunately disprove the Riemann hypothesis, these zeroes would nevertheless lead to some surprisingly nice properties among the primes.  Most notably, the existence of Siegel zeroes would allow us to prove (among other things) the twin prime conjecture \cite{HB83}, small gaps between general $m$-tuples of primes \cite{WrS}, the existence of large intervals where the Goldbach conjecture is true \cite{MaMe}, a hybrid Chowla and Hardy-Littlewood conjecture \cite{TT}, and results about primes in arithmetic progressions that would allow the modulus $q$ to be greater than $\sqrt x$ \cite{FI03}.  It is this last result that is of interest in the present paper.

In the definitions below, we will assume that $(a,q)=1$.  We recall that Chebyshev's functions are given by
\begin{gather*}
\psi(x)=\sum_{n\leq x}\Lambda(n),
\psi(x,q,a)=\sum_{\substack{n\leq x\\n\equiv a\pmod q}}\Lambda(n),
\end{gather*}
where $\Lambda$ is the von Mangoldt function given by
$$\Lambda(n)=\begin{cases} \log p & if\mbox{ }n=p^k\mbox{ }for\mbox{ }prime\mbox{ }p,\\ 0 & otherwise.\end{cases}$$
In 2003, Friedlander and Iwaniec \cite{FI03} proved the following:

\begin{theorem}[Friedlander-Iwaniec, 2003] \label{FIFI}
Let $\chi$ be a real character mod $D$.  Let $x>D^r$ with $r=554,401$, let $q=x^\theta$ with $\theta<\frac{233}{462}$, and let $(a,q)=1$.  Then
$$\psi(x,q,a)=\frac{\psi(x)}{\phi(q)}\left(1-\chi\left(\frac{aD}{(q,D)}\right)+O\left(L(1,\chi)(\log x)^{r^r}\right)\right).$$
\end{theorem}

Notably, this allows for moduli $q$ that are larger than $x^{\frac 12}$.  In fact, their work actually proves this theorem for the slightly larger region of $$\theta<\frac{58\left(1-\frac 1r\right)}{115}$$
for some very large $r$.  The authors rounded off to an exponent of $\frac{233}{462}$ for the purpose of writing their result more simply.

While this theorem allows for better-than-GRH distribution, it requires a rather extreme Siegel zero.  Indeed, this theorem is only non-trivial if
$$L(1,\chi)=o\left(\left(\log D\right)^{-{554,401}^{554,401}}\right),$$
or when there exists a Siegel zero $\beta$ with
$$1-\beta=o\left(\left(\log D\right)^{-{554,401}^{554,401}-2}\right).$$
since $1-\beta\ll L(1,\chi)\log^2 D$ (see e.g. (4.1) of \cite{FI18}).

However, a recent preprint by Zhang \cite{ZhS} has raised the possibility that such a zero does not exist.  In fact, Zhang's paper posits that there is no $L$-function with
$$L(1,\chi)<\frac{c}{\log^{2022} D}$$
for some explicitly computable constant $c$.  Such a bound, of course, would obsolete Theorem \ref{FIFI}.

\section{Main Theorems}
As such, the aim of the present paper is twofold: we relax the requirements on the Siegel zero in the Friedlander-Iwaniec result (so that the result would not be made obsolete by Zhang's potential breakthrough), and we increase the allowable size for $q$.  Towards these aims, we achieve the following result:
\begin{theorem}\label{MainTheoremi}
Let $\chi$ be a primitive real character mod $D$, and let $x$ be such that $D<x^\eta$ with $0<\eta<\frac{1}{50,000}$.  Moreover, let $q$ be an integer such that $q\leq x^{\theta}$ for $\theta<\frac{30}{59}-\varepsilon'$ with $\varepsilon'>62\eta$, and let $a$ be an integer for which $(a,q)=1$.  Then
$$\psi(x,q,a)=\frac{\psi(x)}{\phi(q)}\left(1-\chi\left(\frac{aD}{(q,D)}\right)+O\left(L(1,\chi)\log^7 x\right)\right).$$
This is non-trivial if $L(1,\chi)\log^{7} x=o(1)$.
\end{theorem}
A new idea (in the context of Siegel zeroes) in this paper is that one can improve the level of distribution if one is willing to settle for an Elliott-Halberstam-type result.  As is standard, we let $n\sim x$ denote $x\leq n\leq 2x$.
\begin{theorem}\label{MainTheoremii}
Let $\chi$, $D$, $x$, and $\eta$, be as in Theorem \ref{MainTheoremi}, and let $Q=x^\theta$ with $\theta<\frac{16}{31}-\varepsilon'$ for some $\varepsilon'>62\eta$.
Then
$$\sum_{q\sim Q}\max_{(a,q)=1}\left|\psi(x,q,a)-\left(1-\chi\left(\frac{aD}{(q,D)}\right)\right)\frac{\psi(x)}{\phi(q)}\right|\ll xL(1,\chi)\log^{7} x.$$
\end{theorem}
Notably, since
$$\sum_{\substack{q\sim Q\\ D|q}}\frac{\psi(x)}{\phi(q)}\ll \frac x{\phi(D)},$$
we can rewrite Theorem \ref{MainTheoremii} as a conditional improvement on the Elliott-Halberstam conjecture.
\begin{corollary} Let $\chi$, $D$, $x$, $Q$, and $\varepsilon'$ be as in the previous theorem.  Then
$$\sum_{q\sim Q}\max_{(a,q)=1}\left|\psi(x,q,a)-\frac{\psi(x)}{\phi(q)}\right|\ll xL(1,\chi)\log^{7} x.$$
\end{corollary}

Throughout this paper, we will generally assume that any $\varepsilon<\frac{1}{600}$, as this will be sufficiently small for our purposes.

\section{Ideas for the Paper: Notation}

Let $\chi$ be an exceptional character of conductor $D$, and let $\ast$ denote the Dirichlet convolution.  Moreover, let $\mu$ denote the M\"{o}bius function, and recall that $$\Lambda(n)=(\mu\ast\log)(n),$$ where $\Lambda$ is the von Mangoldt function defined earlier.

Traditionally, questions about primes have tended to focus on the von Mangoldt function.  In \cite{FI03}, the authors' idea was that one can rewrite $\Lambda$ with $$\Lambda=\mu\ast\log\ast \chi\ast \chi\mu,$$ since $(\chi\ast \chi\mu)(n)$ is 1 if $n=1$ and zero otherwise.  Regrouping these terms, one has that
\begin{gather}\label{nulambda}
\Lambda=(\log \ast \chi)\ast(\mu\ast \chi\mu).
\end{gather}
The $\log\ast\chi$ term can be evaluated using standard $L$-function, contour integration techniques, and Weil's bound for Kloosterman sums.  Meanwhile, for the $\mu\ast \chi\mu$ term, one can see that $$|(\mu\ast \chi\mu)(n)|\leq (1\ast\chi)(n),$$and this, too, is easier to evaluate than the von Mangoldt function.

More specifically, define
\begin{gather*}
\lambda(n)=(\chi\ast 1)(n),\\
\lambda'(n)=(\chi\ast \log )(n),\end{gather*} and $$\nu(n)=(\mu\ast (\mu \chi))(n).$$
Importantly, under the assumption of a Siegel zero, sums over $\lambda(n)$ are small.  In particular, if $x>D^2$, then
\begin{gather}\label{lambdaid1}
\sum_{d\leq x}\lambda(d)=xL(1,\chi)+O\left(\sqrt{Dx}\right),
\end{gather}
and
\begin{gather}\label{lambdaid2}
\sum_{D^2<d\leq x}\frac{\lambda(d)}{d}\ll L(1,\chi)\log x,
\end{gather}
which are Lemma 5.1 and equation (5.9) of \cite{FI03}, respectively.  For results on Siegel zeroes, these identities are a key point of leverage, as the assumption that $L(1,\chi)$ is small allows one to extract savings from these two bounds.

The relationship between $\lambda'$ and $\Lambda$ can be given by
\begin{gather}\label{lambdabk}
\lambda'=\chi\ast \log =\chi\ast 1\ast \mu \ast \log =\lambda\ast \Lambda,
\end{gather}
and
\begin{gather}\label{Lambdabk}
\Lambda=\mu\ast \log =\chi\ast \chi\mu \ast \mu\ast \log =\nu\ast \lambda'.
\end{gather}
\section{Ideas for the Paper: the Friedlander-Iwaniec Framework}
This last identity can be used to re-express Chebyshev's function:
$$\psi(x,q,a)=\mathop{\sum\sum}\limits_{\substack{d,m\\dm\leq x\\ dm\equiv a\pmod q}}\nu(d)\lambda'(m).$$
Friedlander and Iwaniec split this double sum into two parts: the part where $d$ is small, and the rest.  In the former case, $\lambda'(m)$ can be evaluated directly with $\nu(d)$ having little impact, and it is this sum over $\lambda'(m)$ that gives the main term in their theorem.

In the case where $d$ is not small, the authors use the fact that we mentioned in our discussion of (\ref{nulambda}), namely that
\begin{gather}\label{lambdanubound}
|\nu(d)|\leq \lambda(d).
\end{gather}
This allows them to write
$$|(\nu\ast \lambda')(n)|\leq (\lambda\ast \lambda')(n)\leq \log(n)(\lambda\ast 1\ast 1)(n).$$
Unfortunately, there are few results that can help with an expression such as the one on the right, since this is a quaternary divisor function $\chi\ast 1\ast 1\ast 1$, and the only divisor functions where $q$ can be taken larger than $\sqrt x$ are binary ones like $1\ast 1$ or ternary ones like $1\ast 1\ast 1$.  To combat this, the authors use an inequality of Landreau \cite{Land} that essentially simplifies the expression to $\lambda\ast 1$.  This simplified expression is much more pliable to ternary sum methods, but this technique comes at the cost of a significantly worse bound and thus only helps under the assumption of a strong Siegel zero.

Here, we handle the sum when $d$ is not small a bit more delicately.  Below, let $\tau_k$ denote the $k$-fold divisor function, and let $\tau=\tau_2$. For some small $\alpha>0$, define $D^*=D^{2+\alpha}$.  By (\ref{lambdabk}), (\ref{lambdanubound}), and the non-negativity of $\lambda$, we can write
\begin{gather}\label{nutolam} \mathop{\sum\sum}\limits_{\substack{m,d \\ d>D^* \\dm\equiv a\pmod q}}\nu(d)\lambda'(m)=\mathop{\sum\sum\sum}\limits_{\substack{ m_1,m_2 \\ d>D^* \\ dm_1m_2\equiv a\pmod q}}\nu(d)\Lambda(m_1)\lambda(m_2)\end{gather}
and
\begin{gather}\label{nutolam2}
\left|\mathop{\sum\sum\sum}\limits_{\substack{ m_1,m_2 \\ d>D^* \\ dm_1m_2\equiv a\pmod q}}\nu(d)\Lambda(m_1)\lambda(m_2)\right|\leq \log x\mathop{\sum\sum\sum}\limits_{\substack{ m_1,m_2 \\ d>D^* \\ dm_1m_2\equiv a\pmod q}}\lambda(d)\lambda(m_2).
\end{gather}
If $dm_1$ is larger than $q^{1+\varepsilon}$ and $m_2>D^*$, we can let $n=dm_1$.  Since $\lambda\ast 1\leq \tau_3$, we can simply apply the congruence condition and Shiu's Brun-Titchmarsh theorem \cite{Sh} to $n$ and still use (\ref{lambdaid2}) on $m_2$ to find our savings.  (Shiu's theorem gives a useful bound for $\tau$ and $\tau_3$ in arithmetic progressions - we discuss this below.)  If $m_1m_2\geq q^{1+\varepsilon}$ then we can let $n=m_1m_2$ and follow the same logic.  If $m_2$ is small then our remaining sum over $\lambda\ast 1$ is the same as the sum that Friedlander and Iwaniec considered in their paper.

In the remaining case, $d$ and $m_2$ are relatively large (i.e. larger than $\frac{x}{q^{1+\varepsilon}}$), yet neither is larger than $q^{1+\varepsilon}$.  To handle this case, we show that most of the information in this sum comes when both $d$ and $m_2$ have large prime factors.  Using exponential sums to detect the congruence, we can then apply recent theorems about Kloosterman sums over primes.  Define
$$S(x,q,a)=\sum_{p\sim x}e\left(\frac{a\bar p}{q}\right).$$
In the case of a single $q$, we apply a result of Fouvry and Shparlinski \cite[Theorem 3.1]{FS}, while in the case where we sum over a range of primes, we apply a result of Irving \cite[Theorem 1.1]{Ir}.
\begin{theorem}[Fouvry-Shparlinski 2011]\label{FSbound} For any integers $a$ and $q$ with $x^\frac 34 \leq q\leq x^\frac 43$ and $(a,q)=1$, and for any $\kappa>0$,
\begin{gather*}
|S(x,q,a)|\ll x^{\frac{15}{16}+\kappa}+x^\frac 23q^{\frac 14+\kappa}.\end{gather*}
\end{theorem}
\begin{theorem}[Irving 2014]\label{Irbound} For any integer $q$ with $x^\frac 23 \leq Q\leq x^\frac 32$, and for any $\kappa>0$,
\begin{gather*}
\sum_{q\sim Q}\max_{(a,q)=1}|S(x,q,a)|\ll \left(Q^\frac 54x^\frac{5}{8}+Qx^\frac{9}{10}+Q^\frac 76x^{\frac{13}{18}}\right)x^\kappa.\end{gather*}
\end{theorem}


\section{The Partition}\label{Partition}
To put our process a bit more concretely, fix a small constant $\alpha>0$.  It will be sufficient to take $\alpha<.00001$.  Again, we define the parameter
$$D^*=D^{2+\alpha}.$$
For a given natural number $x$ and modulus $q$, define the sets
\begin{gather*}\
\mathcal D_4=\{z:\exists z_1\mbox{ such that }z_1^2|z,\mbox{ }z_1^2>\frac{x}{q^{1+\alpha}}\},\\
\mathcal D_5=\{z\not\in \mathcal D_4:\exists z_1\mbox{ such that }z_1^2|z,\mbox{ }D^\alpha <z_1^2\leq \frac{x}{q^{1+\alpha}}\},\\
\mathcal D_6=\{z\not\in \mathcal D_4\cup \mathcal D_5:\exists z_1'\mbox{ such that }z_1'|z,\mbox{ }D^2<z_1\leq x^\frac{3}{10}\}.
\end{gather*}
It will be useful to work with the portions of $d$ and $m_2$ that are square-free, and this partitioning will help us to keep track of square-free parts.  It will also make computations easier if either $d$ or $m_2$ has a medium-sized factor, and $\mathcal D_6$ will separate out these cases.

From here, we define the following partition for the set of natural numbers $dm\leq x$ with $m=m_1m_2$:
\begin{gather*}
I_0=\{(d,m_1,m_2):d\leq D^*\},\\
I_1=\{(d,m_1,m_2)\not \in I_0:m_2\leq D^*\},\\
I_2=\{(d,m_1,m_2)\not \in \bigcup_{i=0}^1 I_i:dm_1\geq q^{1+\alpha}\},\\
I_3=\{(d,m_1,m_2)\not \in \bigcup_{i=0}^2 I_i:m_1m_2\geq q^{1+\alpha}\},\\
I_4=\{d,m_1,m_2\not \in \bigcup_{i=0}^3 I_i:\mbox{at least one of }d,m_2\in \mathcal D_4\},\\
I_5=\{d,m_1,m_2\not \in \bigcup_{i=0}^4 I_i:\mbox{at least one of }d,m_2\in \mathcal D_5\},\\
I_6=\{d,m_1,m_2\not \in \bigcup_{i=0}^5 I_i:\mbox{at least one of }d,m_2\in \mathcal D_6\},\\
I_7=\{d,m_1,m_2\not \in \bigcup_{i=0}^6 I_i\}.
\end{gather*}
We can then define
$$Z_j=\sum_{\substack{(d,m_1,m_2)\in I_j \\ dm_1m_2\equiv a\pmod q}}\nu (d)\Lambda(m_1)\lambda(m_2)$$
and
$$Z_j^*=\sum_{\substack{(d,m_1,m_2)\in I_j \\ (dm_1m_2,q)=1}}\nu(d)\Lambda(m_1)\lambda(m_2).$$
Recalling the relations in (\ref{nutolam}) and (\ref{nutolam2}), we also define
$$T_j=\sum_{\substack{(d,m_1,m_2)\in I_j \\ dm_1m_2\equiv a\pmod q}}\lambda (d)\lambda(m_2)$$
and
$$T_j^*=\sum_{\substack{(d,m_1,m_2)\in I_j \\ (dm_1m_2,q)=1}}\lambda(d)\lambda(m_2).$$
We can use these to bound $Z_j$, as we showed earlier that
\begin{gather}\label{ZT}|Z_j|\ll T_j\log x.
\end{gather}

Since $Z_j$ and $T_j$ depend on $a$ and $q$, we will sometimes denote these sums by $Z_j(a,q)$ and $T_j(a,q)$ (and likewise for $Z_j^*(q)$ and $T_j^*(q)$).  When the context is clear, we will drop the parenthesized term.

The first sum $Z_0$ is the main term.  By Proposition 4.3 of \cite{FI03}, this evaluates to the following.
\begin{theorem}\label{FImainterm} If $x>qD$, then for any $\varepsilon'>0$,
$$Z_0=\frac{\psi(x)}{\phi(q)}\left(1-\chi\left(\frac{aD}{(q,D)}\right)+O\left(L(1,\chi)\log^3 x+\frac{D^\frac 54q^\frac 34}{\sqrt x}x^{\varepsilon'}\right)\right).$$
\end{theorem}
This is non-trivial if $q<x^{\frac 23-\varepsilon}$ for an $\varepsilon$ determined by $\varepsilon'$ and $D$.  Notably, since our Main Theorems require $q<x^{\frac{16}{31}-\varepsilon}$, we see that the latter error term is inconsequential to our work.

For the remaining $j$, we turn our focus from $Z_j$ to $T_j$.

The sum $T_1$ is essentially a ternary divisor sum twisted by our character $\chi$.  Friedlander and Iwaniec (using their result in \cite{FI85}) found that when $q\leq x^{\frac{58(1-\frac 1r)}{115}}$ with $r=554,401$,
$$T_1=(1+o(1))\frac{1}{\phi(q)}T^*_1.$$
Here, we are able to increase the bound on $q$ to $q\leq x^{\frac{30}{59}-\varepsilon}$ for individual $q$ and $q\leq x^{\frac{16}{31}-\varepsilon}$ over a range of $q$.  This is the reason for the increased level of distribution in Theorems \ref{MainTheoremi} and \ref{MainTheoremii}.  We explain the ideas behind this improvement in Section \ref{DivSumEH}.

The sum $T_7$ is the sum where we must use the \cite{FS} and \cite{Ir} results above.  The bounding of $T_2$ through $T_6$, meanwhile, is a fairly straightforward application of (\ref{lambdaid2}) and the multiplicativity of $\lambda$.

Eventually, we will find the following inequalities.
\begin{theorem}\label{Tsforall}
For $x$ and $q$ as above, if $\sqrt x <q< x^{\frac{30}{59}-\varepsilon'}$ for some $\varepsilon'$ such that $\varepsilon'>62\eta$,
\begin{gather*}
T_1\ll \frac{x}{\phi(q)}L(1,\chi)\log^2 x\log\log x+\frac{x}{q}L(1,\chi)^2\log^3x,\\
T_2\ll \frac{x}{q}\log^3 xL(1,\chi),\\
T_3\ll \frac{x}{q}\log^3 xL(1,\chi),\\
T_4\ll q^{\frac 12+\frac{3\alpha}{4}},\\
T_5\ll \frac x{qD^{\frac{\alpha}{4}}}\log^4 x,\\
T_6\ll \frac xqL(1,\chi)\log^5 x,\\
T_7\ll \frac{x}{q}L(1,\chi)\log^6x.
\end{gather*}
If $\sqrt x<Q<x^{\frac{16}{31}-\varepsilon'}$ with $\varepsilon'>62\eta$, then $T_2$ through $T_6$ are the same for each $q$ in the range $[Q,2Q]$, while
$$\sum_{q\sim Q}\max_{(a,q)=1}|T_1(a,q)|\ll xL(1,\chi)\log^2 x\log\log x+xL(1,\chi)^2\log^3x,$$
$$\sum_{q\sim Q}\max_{(a,q)=1}|T_7(a,q)|\ll xL(1,\chi)\log^6 x,$$
\end{theorem}
From these statements and the relationship between $Z_j$ and $T_j$, Theorems \ref{MainTheoremi} and \ref{MainTheoremii} clearly follow.  Note that the largest power of log here is $\log^6x$ (in $T_7$), which, by (\ref{ZT}), means that the power of log in the main theorems must be $\log^7x$.
\section{Divisor Sums}\label{DivSumEH}

Both \cite{FI03} and the current paper depend heavily on results about binary and ternary divisor sums in arithmetic progressions.

The classical result about binary divisor sums in arithmetic progressions was proved in the 1950's by Selberg and Hooley - see \cite{Ho}, \cite[p. 234-237]{Se}, \cite[Corollary 1]{HB79}.

\begin{theorem}\label{binarydivisor}
If $q\leq x^{\frac 23-2\varepsilon}$ and $(a,q)=1$, then
$$\sum_{\substack{n\leq x\\ n\equiv a\pmod q}}\tau(n)=\frac{1}{\phi(q)}\sum_{\substack{n\leq x\\ (n,q)=1}}\tau(n)+O\left(\frac{x}{q^{1+\varepsilon}}\right).$$
\end{theorem}
The first result on ternary sums on arithmetic progressions that moved beyond the square-root barrier came from Friedlander and Iwaniec in 1985 \cite[p. 339]{FI85}.
\begin{theorem}[Friedlander-Iwaniec, 1985]\label{FIternary}
If $q<x^{\frac{58}{115}}$ and $(a,q)=1$ then
$$\sum_{\substack{n\leq x\\ n\equiv a\pmod q}}\tau_3(n)=\frac{1}{\phi(q)}\sum_{\substack{n\leq x\\ (n,q)=1}}\tau_3(n)+O\left(\frac{x}{q^{1+\varepsilon}}\right).$$
\end{theorem}
Since $\tau(n)=(1\ast 1)(n)$ and $\tau_3(n)=(1\ast 1\ast 1)(n)$, it seems logical to hope that we can generalize these results to other convolutions such as $(\chi \ast 1)(n)$ and $(\chi \ast 1\ast 1)(n)$.  A key step in \cite{FI03} is the realization that this can indeed happen.  In fact, equation (5.6) in \cite{FI03} gives the following.
\begin{theorem}[Friedlander-Iwaniec, 2003]\label{FIternary2}
If $x^{\frac{92}{185}}<q<x^{\frac{58}{115}}$ and $(a,q)=1$, then
$$\sum_{\substack{dm\leq x \\ d>D^*\\dm\equiv a\pmod q}}\lambda(d)=\frac{1}{\phi(q)}\sum_{\substack{dm\leq x \\ d>D^*\\ (dm,q)=1}}\lambda(d)+O\left(Dx^{\frac{271}{300}+\varepsilon}q^{-97/120}\right).$$
\end{theorem}
This sum provides the bound for $\theta$ in \cite{FI03}, as all of the other sums in that paper give a looser bound than this for $q$.  Thus, we see that if one is to improve the level of distribution (i.e. increase the value of $\theta$ such that the theorem holds for $q<x^\theta$), one must somehow improve this ternary divisor sum result.

One of the key ideas in Friedlander and Iwaniec's proof of Theorem \ref{FIternary2} was the bounding of the Kloosterman sum
$$K_\chi(a)=\sum_{1\leq h\leq H}\sum_{1\leq n\leq N}\sum_{1\leq m\leq M}e\left(-\frac{ah\bar m\bar n}{q}\right)\chi(n).$$
In particular, the authors found two bounds for this sum.  We will use one of their two bounds, specifically the one below that is denoted (2.6) in their paper.
\begin{theorem}[Friedlander-Iwaniec, 1985]\label{FI:2.6}
Let $(a,q)=1$.  Then
\begin{align*}K_\chi(a)
\ll & Dx^\varepsilon \left(q^\frac 34 H^{\frac 12}M^{\frac 12} + q^\frac 14HM + q^\frac 13H^\frac 23M^\frac 13N^\frac 23+HM^\frac 23N^\frac 23+q^{-1}HMN\right).
\end{align*}
\end{theorem}
Note that the authors originally proved this in \cite{FI85} without the character, and the bound that they found did not have the additional $D$ in the front.  The later adaptation of this result to the case where one of the variables is twisted by $\chi$ costs an additional factor of $D$ (as noted in the discussion just above (5.6) in \cite{FI03}).

The authors of \cite{FI85} also found another bound for $K_\chi(a)$ (which they denoted 2.5).  In lieu of this, however, we use a newer result by Shparlinski, which is Theorem 1.1 in \cite{Shp}.

\begin{theorem}[Shparlinski, 2019]\label{Shp:T1} Let $(a,q)=1$.  Then
$$K_\chi(a)\ll \left(HM+(HM)^\frac 34Q^\frac 14\right)\left(N^\frac 78Q^{-\frac 18}+N^\frac 12\right)Q^{o(1)}.$$
\end{theorem}
If we are willing to consider a range of $q$, we can instead apply Theorem 1.2 of \cite{Shp}:
\begin{theorem}[Shparlinski, 2019]\label{Shp:T2}
Let $\kappa>0$ be a fixed real number, and let $Q$ be sufficiently large.  For all but at most $Q^{1-4\kappa+o(1)}$ values of $q\in [Q,2Q]$,
\begin{gather}\label{goodq}\max_{(a,q)=1}|K_\chi(a)|\ll \left(HM+(HM)^\frac 34Q^\frac 14\right)\left(NQ^{-\frac 14}+N^\frac 12\right)Q^{\kappa+o(1)}.
\end{gather}
\end{theorem}
Using these new bounds of Shparlinski, we will find a slightly different, but slightly stronger, version of Theorem \ref{FIternary2}, proving the following.
\begin{theorem}\label{terdivsum}
For any $\varepsilon>0$, let $\sqrt x<q<x^{\frac{30}{59}-\varepsilon'}$ for any $\varepsilon>0$ with $\varepsilon'>60\eta$ and $(a,Dq)=1$.  Then
$$\sum_{\substack{dm\leq x \\ d>D^*\\dm\equiv a\pmod{Dq}}}\lambda(d)=\frac{1}{\phi(q)}\sum_{\substack{dm\leq x \\ d>D^*\\ (dm,Dq)=1}}\lambda(d)+O\left(\frac{xL(1,\chi)}{Dq\log^8 x}+\frac{x}{Dq}L(1,\chi)^2\log^2 x\right).$$
Additionally, let $\sqrt x<Q<x^{\frac{16}{31}-\varepsilon'}$ for any $\varepsilon>0$ with $\varepsilon'>60\eta$.  Then
$$\sum_{q\sim Q}\max_{(a,Dq)=1}\left|\sum_{\substack{dm\leq x \\ d>D^*\\dm\equiv a\pmod{Dq}}}\lambda(d)-\frac{1}{\phi(q)}\sum_{\substack{dm\leq x \\ d>D^*\\ (dm,Dq)=1}}\lambda(d)\right|\ll \frac{xL(1,\chi)}{D\log^8 x}+\frac xDL(1,\chi)^2\log^2 x.$$
\end{theorem}
Obviously, it would not be difficult to remove the $D$ from the congruences and relative primality statements above.  We do not do so here since it unnecessarily complicates things (as one must consider the impact of the common factors of $D$ and $q$), and also because it is unhelpful for our result.

With this new bound on the level of distribution of twisted ternary divisor sums, we will be able to show the bounds for $T_1$ in Theorem \ref{Tsforall}.  We will prove Theorem \ref{terdivsum} in Sections \ref{binsumsec} to \ref{T1sec}.

\section{Remarks}
This proof proceeds in three main parts.  The first part, which is Section \ref{easysec}, uses Shiu's bound for multiplicative functions to evaluate $T_2$ through $T_6$.  The second part, which is Sections \ref{Kloosec1} to \ref{T7sec}, uses exponential sums and Kloosterman sums to evaluate $T_7$.  Sections \ref{Kloosec1} and \ref{Kloosec2} applies the Kloosterman theorems above to binary sums where one of the terms is prime, and then Section \ref{T7sec} applies this work to the quantity $T_7.$

The third part, which is Sections \ref{binsumsec} to \ref{T1sec}, examines twisted ternary sums in arithmetic progressions and then uses the results of this examination to prove the bounds for $T_1$.  In Section \ref{binsumsec}, we handle the easiest cases where two of the variables are small, and we establish some identities that will help us for later. Sections \ref{binsumsec1a} and \ref{Binsumsec2} examine the case where one of the variables is small, and we prove a result about twisted binary sums that we then apply to this case.  Sections \ref{tersumsec1} to \ref{tersumsec3} give an adapted version of Section 3 of \cite{FI85} that will allow us to use Theorems \ref{FI:2.6} through \ref{Shp:T2} above.  Sections \ref{M2sec} to \ref{ShpQ} apply these theorems to find better levels bounds for the exponential sums that arise from Friedlander and Iwaniec's methods.  Finally, section \ref{T1sec} applies the results of all of these sections to prove Theorem \ref{terdivsum}, bound $T_1$, and complete the proof.

A few remarks:

- We note here that Theorem \ref{FIternary} is not the optimal known result for ternary divisor sums.  Indeed, Heath-Brown \cite{HB86} improved the exponent to $\frac 12+\frac{1}{82}$ for individual $q$ and $\frac 12+\frac{1}{42}$ over a range of $q$. Fouvry, Kowalski, and Michel \cite{FKM} later proved exponents of $\frac 12+\frac{1}{46}$ in the case that $q$ is prime and $\frac 12+\frac{1}{34}$ when $q$ is averaged over a fixed residue class, while a recent preprint of Sharma \cite{Sha} raises the exponent to $\frac 12+\frac{1}{30}$ for individual $q$ in the case where $q$ is square-free or an odd prime power.  However, it is not yet known how to adapt \cite{HB86}, \cite{FKM}, or \cite{Sha} to more general divisor sums like ours.

By contrast, \cite{FI85} and \cite{Shp} adapt more easily to the introduction of $\chi$.  Applying $\chi$ to \cite{FI85} only costs us an additional multiple of $D$ on the latter part of the error term (as Friedlander and Iwaniec themselves note in Proposition 5 of \cite{FI03}), while \cite{Shp} is actually stated such that one could insert any 1-bounded function into the sum with no change in the bound.  


- We also note here that since we have essentially turned a problem about primes into one about divisor sums, we are able to find the following version of the Brun-Titchmarsh theorem if one assumes Siegel zeroes:
\begin{CorBT}
If $q\leq \frac{x^{\frac 23-\varepsilon}}{D^\frac 32}$ then
\begin{gather}\label{BTTH}
\sum_{\substack{ p\leq x \\  p\equiv a \pmod q \\ \chi(p)=1}}1\ll \frac{x}{q}L(1,\chi)\log x.
\end{gather}
\end{CorBT}
This holds because we can bound the characteristic function on such primes with the binary divisor sum-like function $\lambda$, which allows us to use the higher level of distribution $\frac 23-2\varepsilon$ that we saw in Theorem \ref{binarydivisor}.  To our knowledge, no similar bound for primes under the existence of Siegel zeroes appears in the literature.  The closest appears to be that of Maynard \cite{MaBT}, who found that an inequality like (\ref{BTTH}) holds when $q\leq x^\frac 17$ (see Proposition 5 of that paper).

\section{Multiplicativity and the Easier Sums}\label{easysec}

In this section, we handle the easiest sums, namely $T_2$ through $T_6$.  Here, we repeatedly exploit Shiu's theorem for multiplicative functions \cite{Sh}, which states the following.
\begin{theorem}
Let $f$ be a nonnegative multiplicative function for which there exist constants $A_1$ and $A_2$ such that

(i) for any prime $p$ and any natural number $k$,
$$f(p^k)\leq A^k,$$

(ii) for every $\varepsilon>0$ and every natural number $n$,
$$f(n)\leq A_2n^\varepsilon.$$

Additionally, for some small $\alpha,\beta$ with $0<\alpha,\beta<\frac 12$, let $x$, $y$, $q$, and $a$ be integers with $y>q^{1+\alpha}$, $x\geq y>x^\beta$, and $(a,q)=1$.  Then
\begin{gather*}
\sum_{\substack{x-y<n\leq x \\ n\equiv a\pmod q}}f(n)\ll \frac{x}{\phi(q)\log x}\exp\left(\sum_{\substack{p\leq x\\p\nmid q}}\frac{f(p)}{p}\right).
\end{gather*}
\end{theorem}
In most instances, we will take $y\gg x$ as well, which means that the bound on $q$ can be expressed as $x>q^{1+\alpha}$.

In the case where $f=\tau_k$, this gives
\begin{gather}
\sum_{\substack{x-y<n\leq x \\ n\equiv a\pmod q}}\tau_k(n)\ll \frac{x}{q}\log^{k-1} x.
\end{gather}
Trivially, we also have
\begin{gather}
\sum_{n\leq x}\frac{\tau_k(n)}{n}\ll \left(\sum_{n\leq x}\frac 1n\right)^k\ll \log^k x.
\end{gather}
With these, we now handle our easier $T_i$:

\begin{theorem}
$$T_2\ll \frac xq\log^3 xL(1,\chi),$$
and
$$T_3\ll \frac xq\log^3 xL(1,\chi).$$
\end{theorem}
\begin{proof}
We prove for $T_2$; the proof is identical for $T_3$.  Since $(\lambda\ast 1)(n)\leq \tau_3(n)$, we let $r=dm_1$ and find
$$T_2\ll \sum_{D^*<m_2\leq \frac{x}{q^{1+\alpha}}}\lambda(m_2)\sum_{\substack{dm_1\leq \frac{x}{m_2} \\ dm_1\equiv a\overline{m_2}\pmod q}}\lambda(d)\ll \sum_{D^*<m_2\leq \frac{x}{q^{1+\alpha}}}\lambda(m_2)\sum_{\substack{r\leq \frac{x}{m_2} \\ r\equiv a\overline{m_2}\pmod q}}\tau_3(r).$$
As $r>q^{1+\alpha}$, we can invoke Shiu's theorem and (\ref{lambdaid2}) to find that
$$\ll \frac xq\log^2 x\sum_{D^*<m_2\leq \frac{x}{q^{1+\alpha}}}\frac{\lambda(m_2)}{m_2}\ll \frac xq\log^3 xL(1,\chi).$$
\end{proof}

In the next several sections, we will need to keep track of the square and squarefree parts of $d$ and $m_2$.  Hence, we will write
\begin{gather*}
d=d_sd',\\
m_2=m_sm',
\end{gather*}
where $d_s$ and $m_s$ are squares and $d'$ and $m'$ are square-free.

\begin{theorem} If $q>\sqrt x$ then
$$T_4\ll q^{\frac 12+\frac{3\alpha}{4}}.$$
\end{theorem}
\begin{proof}
We will assume that $d\in \mathcal D_4$; the proof is identical if $m_2\in \mathcal D_4$.  Let $r=d'm_1m_2$.  So
\begin{align*}
T_4\ll &\mathop{\sum\sum\sum\sum}\limits_{\substack{d_sd'm_1m_2\leq x \\ d_sd'm_1m_2\equiv a\pmod q \\ \frac{x}{q^{1+\alpha}}<d_s}}\tau(d_s)\tau(d')\tau(m_2)\\
\ll &x^{\frac{\alpha}{5}}\sum_{r\leq q^{1+\alpha}}\sum_{\substack{\frac{x}{q^{1+\alpha}}<d_s\leq \frac xr \\ d_s\equiv a\bar r\pmod q }}1,
\end{align*}
since $\tau(n)\ll x^{o(1)}$.  Note that we must have $d_s\leq q^{1+\alpha}$, else we would have $dm_1>q^{1+\alpha}$ and hence this triple would have been in $I_2$.  So $r\geq \frac{x}{q^{1+\alpha}}$.

Using the notation that 
$$e_z(y)=e^{\frac{2\pi i y}{z}},$$
we express the congruence using exponential sums:
\begin{align*}
&x^{\frac{\alpha}{5}}\mathop{\sum\sum}\limits_{\substack{d_sr\leq x \\ d_sr\equiv a\pmod q \\ \frac{x}{q^{1+\alpha}}<d_s\leq q^{1+\alpha}}}1\\
&=\frac{x^{\frac{\alpha}{5}}}{q}\sum_{v=1}^{q-1}\mathop{\sum\sum}\limits_{\substack{d_sr\leq x \\ \frac{x}{q^{1+\alpha}}<d_s\leq q^{1+\alpha}}}e_q\left(av\overline{d_s}-vr\right)+\frac{x^{\frac{\alpha}{5}}}{q}\mathop{\sum\sum}\limits_{\substack{d_sr\leq x \\ \frac{x}{q^{1+\alpha}}<d_s\leq q^{1+\alpha}}}1\\
\end{align*}
For the latter double sum, we let $d_s=u^2$ and see that
\begin{align*}\frac{x^{\frac{\alpha}{5}}}{q}\mathop{\sum\sum}\limits_{\substack{d_sr\leq x \\ \frac{x}{q^{1+\alpha}}<d_s\leq q^{1+\alpha}}}1=&\frac{x^{\frac{\alpha}{5}}}{q}\sum_{\frac{\sqrt x}{q^{\frac 12+\frac \alpha 2}}<u\leq q^{\frac 12+\frac \alpha 2}}\sum_{r\leq \frac{x}{u^2}}1\\
=&\frac{x^{\frac{\alpha}{5}}}{q}\sum_{\frac{\sqrt x}{q^{\frac 12+\frac \alpha 2}}<u\leq q^{\frac 12+\frac \alpha 2}}\frac{x}{u^2}\\
\ll &\frac{x^{\frac 12+\frac{\alpha}{5}}}{q^{\frac 12-\frac \alpha 2}}\ll q^{\frac 12+\frac{3\alpha}{4}},
\end{align*}
since $q>\sqrt x$.

For the remaining sum,
\begin{align*}
\frac{x^{\frac{\alpha}{5}}}{q}&\sum_{v=1}^{q-1}\mathop{\sum\sum}\limits_{\substack{d_sr\leq x \\ \frac{x}{q^{1+\alpha}}<d_s\leq q^{1+\alpha}}}e_q\left(av\overline{d_s}-vr\right)\\
&\ll \frac{x^{\frac{\alpha}{5}}}{q}\sum_{v=1}^{q-1}\sum_{\frac{x}{q^{1+\alpha}}<d_s\leq q^{1+\alpha}}\frac{1}{||v/q||}\\
&=x^{\frac{\alpha}{4}}\sum_{\frac{x}{q^{1+\alpha}}<d_s\leq q^{1+\alpha}}1\\
&\ll x^{\frac{\alpha}{4}}q^{\frac 12+\frac{\alpha}{2}},
\end{align*}
where $||\cdot ||$ indicates the distance to the nearest integer.  Again, we use the fact that $q>\sqrt x$ to write
$$x^{\frac{\alpha}{4}}q^{\frac 12+\frac{\alpha}{2}}<q^{\frac 12+\frac{3\alpha}{4}},$$and hence the lemma follows.

\end{proof}
\begin{theorem}
$$T_5\ll \frac x{qD^{\frac{\alpha}{4}}}\log^4 x.$$
\end{theorem}
\begin{proof}
As above, we will assume that $d\in \mathcal D_5$; the proof is identical if $m_2\in \mathcal D_5$.  So
\begin{align*}
T_5\ll &\mathop{\sum\sum\sum\sum}\limits_{\substack{d_sd'm_1m_2\leq x \\ d_sd'm_1m_2\equiv a\pmod q \\ D^\alpha<d_s<\frac{x}{q^{1+\alpha}}}}\tau(d_s)\tau(d')\tau(m_2).
\end{align*}
Again, let $r=d'm_1m_2$.  Since $D^\alpha<d_s<\frac{x}{q^{1+\alpha}}$, we know that $r>q^{1+\varepsilon}$, and hence
\begin{align*}
T_5\ll &\sum_{D^\alpha<d_s<\frac{x}{q^{1+\alpha}}}\tau(d_s)\sum_{\substack{r\leq \frac{x}{d_s} \\ r\equiv a\overline{d_s}\pmod q}}\tau_5(r)\\
\ll &\frac xq\log^4 x\sum_{D^\alpha<d_s<\frac{x}{q^{1+\alpha}}}\frac{\tau(d_s)}{d_s}.
\end{align*}
Letting $d_s=u^2$, and noting that $\tau(u^2)\leq u^\frac \alpha 4$, we have
\begin{align*}
T_5\ll &\frac xq\log^4 x\sum_{D^{\frac{\alpha}{2}}<u<\frac{x}{q^{1+\alpha}}}\frac{1}{u^{2-\frac{\alpha}{4}}}\\
\ll &\frac x{\phi(q)D^{\frac{\alpha}{4}}}\log^4 x.
\end{align*}
\end{proof}

\begin{theorem}\label{T6Bd} If $q\leq x^{\frac{7}{10}-\varepsilon}$ for any $\varepsilon>2\eta$, then
$$T_6\ll \frac xqL(1,\chi)\log^5 x.$$
\end{theorem}
\begin{proof}
Again, assume that $d\in \mathcal D_6$; the proof works the same if $m_2\in \mathcal D_6$.  Hence $d_s\leq D^\alpha$.  We claim that
\begin{gather}\label{split}
\lambda(d)\leq \lambda(d_s)\lambda(d').
\end{gather}
To show this, we first recall that for any prime, we either have $\chi(p)=\pm 1$ or 0.  If $\chi(p)=-1$ then $\lambda(p^{2k})=1$ and hence
$$\lambda(p^{2k+1})=0=\lambda(p)=\lambda(p)\lambda(p^{2k}).$$
If $\chi(p)=0$ then we also have $\lambda(p^{2k})=1$ and hence
$$\lambda(p^{2k+1})=1=\lambda(p)=\lambda(p)\lambda(p^{2k}).$$
If $\chi(p)=1$, we have
$$\lambda(p^{2k+1})=\tau(p^{2k+1})=2k+2\leq (2k+1)(2)=\tau(p^{2k})\tau(p)=\lambda(p^{2k})\lambda(p).$$
So for a given natural number $d$, write
$$d=p_1^{2k_1+j_1}\cdots p_t^{2k_t+j_t}$$
where $k_i$ is a whole number and $j_i$ is 0 or 1. Since $\lambda$ is multiplicative, we then have
\begin{align*}
\lambda(d)=&\lambda(p_1^{2k_1+j_1}\cdots p_t^{2k_t+j_t})=\lambda(p_1^{2k_1+j_1})\cdots \lambda(p_t^{2k_t+j_t})\\
\leq &\lambda(p_1^{2k_1})\lambda(p_1^{j_1})\cdots \lambda(p_t^{2k_t})\lambda(p_t^{j_t})=\lambda(p_1^{2k_1}\cdots p_t^{2k_t})\lambda(p_1^{j_1}\cdots p_t^{j_t})=\lambda(d_s)\lambda(d').
\end{align*}
This proves (\ref{split}).  So
\begin{align*}
T_6\ll &\mathop{\sum\sum\sum\sum}\limits_{\substack{m_1d_sd'm_2\leq \frac{x}{m_1} \\ d_sd'm_2\equiv a\pmod q \\ d_s<D^\alpha}}\lambda(d_s)\lambda(d')\lambda(m_2).
\end{align*}
Now, by the assumption that $d\in \mathcal D_6$, we can write $d'=jk$ where $D^*<j\leq x^\frac{3}{10}$.  Since $d'$ has no square factors, we know that $(j,k)=1$, and hence by the multiplicativity of $\lambda$ we have
$$\lambda(d')=\lambda(j)\lambda(k).$$
Hence,
\begin{align*}T_6\ll &\sum_{d_s\leq D^\alpha}\lambda(d_s)\sum_{D^*<j\leq x^\frac{3}{10}}\lambda(j)\mathop{\sum\sum\sum}\limits_{\substack{km_1m_2\leq \frac{x}{d_sj} \\ km_1m_2\equiv a\overline{jd_s}\pmod q}}\lambda(k)\lambda(m_2).
\end{align*}
Letting $r=m_1m_2k$, and again letting $u=\sqrt{d_s}$, we know that $$r>x^{\frac{7}{10}}D^{-\alpha}>q^{1+\varepsilon'}$$
for some $\varepsilon'>0$.  So again by Shiu's theorem,
\begin{align*}
T_6\ll &\sum_{d_s\leq D^\alpha}\tau(u^2)\sum_{D^*<j\leq x^\frac{3}{10}}\lambda(j)\sum_{\substack{r\leq \frac{x}{m_1ju^2} \\ r\equiv a\overline{m_1ju^2}\pmod q}}\tau_5(r)\\
\ll &\frac xq\log^4 x\sum_{u\leq D^\frac{\alpha}{2}}\frac{1}{u^{2-\frac{\alpha}{4}}}\sum_{D^*<j\leq x^\frac{3}{10}}\frac{\lambda(j)}{j}\\
\ll &\frac xqL(1,\chi)\log^5 x\sum_{u\leq D^\frac{\alpha}{2}}\frac{1}{u^{2-\frac{\alpha}{4}}}\\
\ll &\frac xqL(1,\chi)\log^5 x.
\end{align*}

\end{proof}

For the remainder of the paper, we can assume that $d_s$ and $m_s$ are smaller than $D^\alpha$.

\section{Kloosterman Sums and Common Factors}\label{Kloosec1}

In the next three sections, we prove lemmas that will allow us to apply the results about Kloosterman sums over primes to $T_7$.  In this section, we prove a lemma about integers that share a large common factor with $q$.  In the next, we use the Fouvry-Shparlinski and Irving results to prove useful lemmas about prime sums in arithmetic progressions.  Finally, in Section \ref{T7sec}, we apply these lemmas to evaluate $T_7$.


We will soon write the congruence modulo $q$ in terms of exponential sums.  However, when we do so, we will need to know how many of the numbers $r$ on the interval $[1,q-1]$ have a large factor in common with $q$.  To this end, we prove the following lemma.

\begin{lemma}\label{bigr}\begin{align*}
\sum_{\stackrel{1\leq r\leq q-1}{(r,q)>q^{\frac 15-\alpha}}}\frac 1r
\ll &q^{-\frac 15+3\alpha}.
\end{align*}
\end{lemma}

\begin{proof}
First, for a given $y<q$, the number of $r<y$ in this summand can be bounded with
$$\#\{1\leq r\leq y:q^{\frac 15-\alpha}<(r,q)\leq q-1\}\leq 2\sum_{\stackrel{j|q}{q^{\frac 15-\alpha}<j\leq y}}\frac yj.$$
Note that this sum over $j$ will be largest when $q$ is smooth.  Specifically, this will be maximized when all of the prime factors of $q$ are less than $K\log q$ for some constant $K$.

So let $\Phi(y,z)$ denote the number of $z$-smooth natural numbers less than $y$.  By Dickman's theorem \cite{Di}, we know that for $u=\log y/\log z$, we have
$\Phi(y,z)\ll yu^{-u}$.
So for any $y\geq q^{\frac 15-\alpha}$,
\begin{align*}\Phi(y,K\log q)\ll & y\left(\frac{\log y}{(1+o(1))\log \log q}\right)^{-\frac{\log y}{(1+o(1))\log\log q}}\\
=&y\left(y^{-\frac{\log\log y-\log \log \log q}{(1+o(1))\log\log q}}\right)\\
=&y^{o(1)}\end{align*}
since $\log\log y/\log \log q=1+o(1)$.  Hence by partial summation,
\begin{align*}
2\sum_{\stackrel{j|q}{q^{\frac 15-\alpha}<r\leq y}}\frac yj=&2y\left(\sum_{q^{\frac 15-\alpha}<j<y}\left(\frac 1j-\frac{1}{j+1}\right)\sum_{\stackrel{j'|q}{q^{\frac 15-\alpha}<j'<j}}1+\frac{1}{q-1}\sum_{\stackrel{j'|q}{q^{\frac 15-\alpha}<j\leq y}}1\right)\\
\ll &2y\left(\sum_{q^{\frac 15-\alpha}<j<q}\frac{j^{o(1)}}{j^{2}}+\frac{y^{o(1)}}{q}\right)\\
\ll &yq^{-\frac 15+\alpha}y^{o(1)}\\
\ll &yq^{-\frac 15+2\alpha}.
\end{align*}
So again by partial summation,
\begin{align*}
\sum_{\stackrel{1\leq r\leq q-1}{(r,q)>q^{\frac 15-\alpha}}}\frac 1r=&\sum_{q^{\frac 15-\alpha}\leq y<q-1}\left(\frac 1y-\frac{1}{y+1}\right)\sum_{\stackrel{q^{\frac 15-\alpha}\leq r\leq y}{(r,q)>q^{\frac 15-\alpha}}}1+\frac{1}{q-1}\sum_{\stackrel{q^{\frac 15-\alpha}\leq r\leq q-1}{(r,q)>q^{\frac 15-\alpha}}}1\\
\ll & q^{-\frac 15+2\alpha}\sum_{q^{\frac 15-\alpha}\leq y<q-1}\frac 1y+q^{-\frac 15+2\alpha}\\
\ll &q^{-\frac 15+3\alpha}.
\end{align*}
\end{proof}

\section{Kloosterman Sums Over Primes}\label{Kloosec2}

We next apply the Kloosterman results in Lemmas \ref{FSbound} and \ref{Irbound} to evaluate binary sums involving primes in arithmetic progressions.

For ease of notation, for some $M,N$ with $MN\leq x$ and some $M_1\leq M$ and $N_1\leq N$, define the intervals
\begin{gather*}
\mathcal M=[M,M+M_1],\\
\mathcal N=[N,N+N_1].
\end{gather*}
Turning our attention now to sums over primes, we have the following:
\begin{lemma}
For any $\varepsilon>0$, let $q<x^{\frac{17}{32}-\varepsilon}$, and let $a$ and $b$ be natural numbers such that $(a,q)=1$ and $(b,D)=1$.  For $M$, $N$, $M_1$, and $N_1$ as defined above, let $\frac{x}{q^{1+\varepsilon}}\leq M\leq \sqrt x$.  Then
$$\sum_{\substack{p\in \mathcal M\\ p\equiv b \pmod D}}\sum_{\substack{n\in \mathcal N\\ np\equiv a\pmod q}}1=\frac{1}{q}\sum_{\substack{p\in \mathcal M \\ p\equiv a \pmod D}}\sum_{n\in \mathcal N}1+O\left(x^{\frac{15}{32}+\kappa'}+x^\frac 13q^{\frac 14+\kappa'}\right)$$
for any $\kappa'>0$.
\end{lemma}
\begin{proof}
Again writing the congruence with exponential sums:
\begin{align*}
\sum_{\substack{p\in \mathcal M\\ p\equiv b \pmod D}}\sum_{\substack{n\in \mathcal N\\ np\equiv a\pmod q}}1=&\frac{1}{Dq}\sum_{r=1}^{q}\sum_{s=1}^D\sum_{p\in \mathcal M}e_q\left(\bar p ra\right)e_D\left(s(b\bar p-1)\right)\sum_{n\in \mathcal N}e_q\left(-nr\right)\\
\end{align*}
We define $N_0$ to be the term with $r=q$.  In other words,
$$N_0=\frac{1}{Dq}\sum_{s=1}^D\sum_{p\in \mathcal M}e_D\left(s(b\bar p-1)\right)\sum_{n\in \mathcal N}1=\frac 1q\sum_{\substack{p\in \mathcal M\\ p\equiv b\pmod D}}\sum_{n\in \mathcal N}1.$$
So
\begin{align*}
\bigg{(}\sum_{\substack{p\in \mathcal M\\ p\equiv b \pmod D}}&\sum_{\substack{n\in \mathcal N\\ np\equiv a\pmod q}}1\bigg{)}-N_0\\
=&\frac{1}{Dq}\sum_{r=1}^{q-1}\sum_{s=1}^D\sum_{p\in \mathcal M}e_q\left(\bar p ra\right)e_D\left(s(b\bar p-1)\right)\sum_{n\in \mathcal N}e_q\left(-nr\right)\\
\ll &\frac{1}{D}\sum_{r=1}^{q-1}\sum_{s=1}^D\frac{1}{||r||}\left|\sum_{p\in \mathcal M}e_{\frac{Dq}{(D,q)}}\left(\bar p \left(\frac{D}{(D,q)}ra+\frac{q}{(D,q)}sb\right)\right)\right|\\
\end{align*}
Let $R$ denote the set of $r$ for which $1\leq r<q$ and
$$\left(r,q\right)>q^{\frac 15-\alpha}.$$
Then
\begin{align*}
\bigg{(}\sum_{\substack{p\in \mathcal M\\ p\equiv b \pmod D}}&\sum_{\substack{n\in \mathcal N\\ np\equiv a\pmod q}}1\bigg{)}-N_0\\
\ll &\frac{1}{D}\sum_{r\in R}\sum_{s=1}^D\frac{1}{||r||}\left|\sum_{p\in \mathcal M}e_{\frac{Dq}{(D,q)}}\left(\bar p \left(\frac{D}{(D,q)}ra+\frac{q}{(D,q)}sb\right)\right)\right|\\
&+\frac 1D\sum_{r\not \in R}\sum_{s=1}^D\frac{1}{||r||}\left|\sum_{p\in \mathcal M}e_{\frac{Dq}{(D,q)}}\left(\bar p \left(\frac{D}{(D,q)}ra+\frac{q}{(D,q)}sb\right)\right)\right|.
\end{align*}
For the sum over $R$, we have
\begin{align*}
\frac{1}{D}&\sum_{r\in R}\sum_{s=1}^D\frac{1}{||r||}\left|\sum_{p\in \mathcal M}e_{\frac{Dq}{(D,q)}}\left(\bar p \left(\frac{D}{(D,q)}ra+\frac{q}{(D,q)}sb\right)\right)\right|
\leq &2M_1\sum_{\substack{r\in R \\ r\leq \frac{q}{2}}}\frac{1}{r}
\ll & M_1q^{-\frac 15+3\alpha}
\end{align*}
by the previous lemma.

For the remaining sum, define
$$j=\left(\frac{D}{(D,q)}ra+\frac{q}{(D,q)}sb,q\right),$$
and write $j=j'j''$ where $j'|\frac{q}{(D,q)}$ and $$\left(j'',\frac{q}{j'(D,q)}\right)=1.$$
Since $j''$, we know that $j''$ must divide $(D,q)$, and hence $j''\leq D$.

Meanwhile, since $j'$ divides $\frac{q}{(D,q)}$, it must also divide $\frac{D}{(D,q)}ra$.  However, $a$ is relatively prime $q$ by assumption and $(r,q)\leq q^{\frac 15-\alpha}$.  Since $j'|q$, this means that $j'\leq Dq^{\frac 15-\alpha}$.  Thus
$$j\leq D^2q^{\frac 15-\alpha}.$$
Thus, the reduced fraction in the exponential term $e_{\frac{Dq}{(D,q)}}\left(\bar p \left(\frac{D}{(D,q)}ra+\frac{q}{(D,q)}sb\right)\right)$ must have denominator at least
$$\frac{Dq}{(q,D)j}\geq \frac{q}{j}\geq \frac{q^{\frac 45+\alpha}}{D^{2}}.$$
Since
$$(qD)^\frac 34<\frac{x}{q^{1+\varepsilon}}<\sqrt x\leq \left(\frac{q^{\frac 45+\alpha}}{D^{2}}\right)^\frac 43,$$
we can apply Theorem \ref{FSbound}, finding that
\begin{align*}
\frac 1D&\sum_{r\not \in R}\sum_{s=1}^D\frac{1}{||r||}\left|\sum_{p\in \mathcal M}e_{\frac{Dq}{(D,q)}}\left(\bar p \left(\frac{D}{(D,q)}ra+\frac{q}{(D,q)}sb\right)\right)\right|\\
\ll & \log q\left(M^{\frac{15}{16}+\kappa}+M^{\frac 23}q^{\frac 14+\kappa}\right)\\
\ll &x^{\frac{15}{32}+\kappa'}+q^\frac 14x^{\frac 13+\kappa'},
\end{align*}
since $M\leq \sqrt x$.
\end{proof}
We have a similar, though slightly stronger, result if we take a sum over $q\sim Q$:
\begin{lemma}\label{Irvcor}
For any $\varepsilon>0$, let $Q<x^{\frac{17}{32}-\varepsilon}$.  Let $\frac{x}{Q^{1+\varepsilon}}\leq M\leq \sqrt x$, and let $M_1$, $N_1$, $N$, $a$, and $b$ be as before.  Then
$$\sum_{q\sim Q}\left|\sum_{\substack{p\in \mathcal M\\ p\equiv b \pmod D}}\sum_{\substack{n\in \mathcal N\\ np\equiv a\pmod q}}1-\frac 1q\sum_{\substack{p\in \mathcal M\\ p\equiv b \pmod D}}\sum_{n\in \mathcal N}1\right|\ll \left(Q^\frac 54x^\frac{5}{16}+Qx^\frac{9}{20}+Q^\frac 76x^{\frac{13}{36}}\right)x^{\kappa'}$$
for any $\kappa'>0$.
\end{lemma}
\begin{proof}
The proof is nearly identical to the previous theorem except that we apply Theorem \ref{Irbound} instead of \ref{FSbound}.
\end{proof}
We can also enlarge the intervals without changing the result:
\begin{corollary}\label{Kloocor}
For any $\varepsilon>0$, let $q<x^{\frac{17}{32}-\varepsilon}$, and let $a$ and $b$ be as before.  Then
$$\sum_{\substack{\frac{x}{q^{1+\alpha}}\leq p\leq \sqrt x \\ p\equiv b \pmod D}}\sum_{\substack{\sqrt x\leq n\leq \frac xp\\ np\equiv a\pmod q}}1=\frac{1}{q}\sum_{\substack{\frac{x}{q^{1+\alpha}}\leq p\leq \sqrt x \\ p\equiv b \pmod D}}\sum_{\sqrt x\leq n\leq \frac xp}1+O\left(x^{\frac{15}{32}+\kappa'}+x^\frac 13q^{\frac 14+\kappa'}\right)$$
for any $\kappa'>0$.
\end{corollary}

\begin{corollary}\label{KloocorIr}
For any $\varepsilon>0$, let $Q<x^{\frac{17}{32}-\varepsilon}$, and let $a$ and $b$ be as before.  Then
$$\sum_{q\sim Q}\left|\sum_{\substack{\frac{x}{Q^{1+\alpha}}\leq p\leq \sqrt x \\ p\equiv b \pmod D}}\sum_{\substack{\sqrt x\leq n\leq \frac xp\\ np\equiv a\pmod q}}1-\frac{1}{q}\sum_{\substack{\frac{x}{q^{1+\alpha}}\leq p\leq \sqrt x \\ p\equiv b \pmod D}}\sum_{\sqrt x\leq n\leq \frac xp}1\right|\ll \left(Q^\frac 54x^\frac{5}{16}+Qx^\frac{9}{20}+Q^\frac 76x^{\frac{13}{36}}\right)x^{\kappa'}$$
for any $\kappa'>0$.
\end{corollary}
Splitting the regions into dyadic intervals proves the corollaries.

\section{The Sum $T_7$}\label{T7sec}
Now, we can turn our attention to $T_7$.  We have the following:

\begin{theorem}For $q<x^\frac{30}{59}$, there exists a $\rho>0$ such that
$$T_7\ll \frac{x}{q}L(1,\chi)\log^2 x\log^4 D+\frac{x^{1-\rho}}{q}.$$
\end{theorem}
Note that this is inequality is actually slightly better than what is claimed in Theorem \ref{Tsforall}, where we wrote $\log^6 x$ instead of $\log^2 x\log^4 D$.  This is of little import if $\frac{\log x}{\log D}\ll 1$ but could yield some marginal savings if $\frac{\log D}{\log x}=o(1)$.  However, since the bound on $T_6$ in Theorem \ref{T6Bd} had a $\log^5x$, we see that the potential savings from this insight is at most a single power of log in Theorems \ref{MainTheoremi} and \ref{MainTheoremii}, and hence we do not pursue this further.

\begin{proof}
Here, we have $d_s,m_s\leq D^\alpha$.  We also have $$\frac{x}{q^{1+\alpha}}\leq d'd_s,m'm_s\leq q^{1+\alpha},$$since the triple is contained in neither $I_2$ nor $I_3$.  This means that
$$m_1\leq \frac{x}{\left(\frac{x}{q^{1+\alpha}}\right)^2}=\frac{q^{2+2\alpha}}{x}.$$

Moreover, we know that $d'$ and $m'$ have no factors between $D^2$ and $x^{\frac{3}{10}}$.  So let $j$ denote the largest integer such that $j|d'$ and $j\leq D^2$.  Then $\frac{d'}{j}$ has no factors less than $x^{\frac{3}{10}}D^{-2}$ besides 1.  Since $d'<x^{\frac{3}{5}}D^{-4}$, this means that $\frac{d'}{j}$ has no factors on the interval $\bigg{(}1,\sqrt{\frac{d'}{j}}\bigg{]}$, which implies that $\frac{d'}{j}$ is prime.

Thus, we write $d'=p_1j$ and $m'=p_2l$, where $j,l\leq D^2$ and $$\frac{x}{q^{1+\alpha}D^{2+\alpha}}\leq p_1,p_2\leq q^{1+\alpha}.$$We will assume that $p_1\leq p_2$; the proof will work the same if $p_2<p_1$.  Hence $$p_1\leq \sqrt{\frac{x}{m_1d_sm_sjl}}.$$
Now,
\begin{align*}
T_7\ll &\sum_{m_1\leq \frac{q^{2+2\alpha}}{x}}\mathop{\sum\sum}\limits_{d_s,m_s\leq D^\alpha}\lambda(d_s)\lambda(m_s)\mathop{\sum\sum}\limits_{l,j\leq D^2}\lambda(l)\lambda(j)\\
&\cdot\sum_{\frac{x}{q^{1+\alpha}D^{2+\alpha}}\leq p_1\leq \sqrt{\frac{x}{m_1d_sm_sjl}}} \lambda(p_1)\mathop{\sum}\limits_{\substack{p_1\leq p_2\leq \frac{x}{m_1m_sd_sjlp_1}\\ p_1p_2\equiv a\overline{jlm_1m_sd_s}\pmod q}}\lambda(p_2),
\end{align*}
where we know that $$\frac{x}{m_1m_sd_sjlp_1}\leq q^{1+\alpha}.$$

We will ignore the outer terms for the moment and focus on the sum over $p_1$ and $p_2$.  To this end, write
$$Y=m_1m_sd_sjl$$
and
$$R=\frac{x}{q^{1+\alpha}D^{2+\alpha}}.$$
So we consider the sum
$$\sum_{R\leq p_1\leq \sqrt{\frac{x}{Y}}} \lambda(p_1)\mathop{\sum}\limits_{\substack{p_1\leq p_2\leq \frac{x}{Yp_1}\\ p_1p_2\equiv a\overline{Y}\pmod q}}\lambda(p_2).$$

Clearly, $\lambda(p)\leq 2$, and $\lambda(p)$ is non-zero only if $\chi(p)=1$ or 0.  Since both of the $p_i$ are primes that are much larger than $D$, we know that $\chi(p_i)\neq 0$ for $i=1,2$.  So
\begin{align*}\sum_{R\leq p_1\leq \sqrt{\frac{x}{Y}}} \lambda(p_1)\mathop{\sum}\limits_{\substack{p_1\leq p_2\leq \frac{x}{Yp_1}\\ p_1p_2\equiv a\overline{Y}\pmod q}}\lambda(p_2)\ll \sum_{\substack{0<b<D \\ \chi(b)=1}}\sum_{\substack{R\leq p_1\leq \sqrt{\frac{x}{Y}}\\ p_1\equiv b\pmod D}} \mathop{\sum}\limits_{\substack{p_1\leq r\leq \frac{x}{Yp_1}\\ p_1r\equiv a\overline{Y}\pmod q}}1,
\end{align*}
where the switch in notation from $p_2$ to $r$ indicates that we have dropped the requirement that $p_2$ be prime.  From Corollary \ref{Kloocor},
\begin{align*}
\sum_{\substack{0<b<D \\ \chi(b)=1}}&\sum_{\substack{R\leq p_1\leq \sqrt{\frac{x}{Y}}\\ p_1\equiv b\pmod D}} \mathop{\sum}\limits_{\substack{p_1\leq r\leq \frac{x}{Yp_1}\\ p_1r\equiv a\overline{Y}\pmod q}}1\\
&=\frac 1q\sum_{\substack{0<b<D \\ \chi(b)=1}}\sum_{\substack{R\leq p_1\leq \sqrt{\frac{x}{Y}}\\ p_1\equiv b\pmod D}} \mathop{\sum}\limits_{p_1\leq r\leq \frac{x}{Yp_1}}1+O\left(x^{\kappa'}\left(\left(\frac{x}{Y}\right)^{\frac{15}{32}}+q^\frac 14\left(\frac{x}{Y}\right)^{\frac{1}{3}}\right)\right)\\
&\ll \frac{1}{2q}\sum_{\substack{R\leq p_1\leq \sqrt{\frac{x}{Y}}}}\lambda(p_1)\mathop{\sum}\limits_{p_1\leq r\leq \frac{x}{Yp_1}}1+O\left(x^{\kappa'}\left(\left(\frac{x}{Y}\right)^{\frac{15}{32}}+q^\frac 14\left(\frac{x}{Y}\right)^{\frac{1}{3}}\right)\right)\\
\end{align*}
For the double sum, we have
\begin{align*}
\frac{1}{2q}\sum_{\substack{R\leq p_1\leq \sqrt{\frac{x}{Y}}}}\lambda(p_1)\mathop{\sum}\limits_{p_1\leq r\leq \frac{x}{Yp_1}}1\ll &\frac{x}{2qY}\sum_{\substack{R\leq p_1\leq \sqrt{\frac{x}{Y}}}}\frac{\lambda(p_1)}{p_1}\ll \frac{x\log xL(1,\chi)}{2qY}.
\end{align*}
So
\begin{align}\label{T7}
T_7\ll &\sum_{m_1\leq \frac{q^{2+2\alpha}}{x}}\mathop{\sum\sum}\limits_{d_s,m_s\leq D^\alpha}\lambda(d_s)\lambda(m_s)\mathop{\sum\sum}\limits_{l,j\leq D^2}\lambda(l)\lambda(j)\left[\frac{x\log xL(1,\chi)}{2qY}+x^{\kappa'}\left(\left(\frac{x}{Y}\right)^{\frac{15}{32}}+q^\frac 14\left(\frac{x}{Y}\right)^{\frac{1}{3}}\right)\right].
\end{align}
Beginning with the first of the bracketed terms, we can let $d_s=u^2$ and $m_s=v^2$, giving
\begin{align*}
\sum_{m_1\leq \frac{q^{2+2\alpha}}{x}}&\mathop{\sum\sum}\limits_{d_s,m_s\leq D^\alpha}\lambda(d_s)\lambda(m_s)\mathop{\sum\sum}\limits_{l,j\leq D^2}\lambda(l)\lambda(j)\frac{x\log xL(1,\chi)}{2qY}\\
\ll &\frac{xL(1,\chi)\log x}{q}\sum_{m_1\leq \frac{q^{2+2\alpha}}{x}}\frac{1}{m_1}\mathop{\sum\sum}\limits_{u,v\leq D^\frac \alpha 2}\frac{1}{u^{2-\alpha}v^{2-\alpha}}\mathop{\sum\sum}\limits_{l,j\leq D^2}\frac{\tau(l)\tau(j)}{lj}\\
\ll &\frac{xL(1,\chi)\log^2 x\log^4 D}{q}.
\end{align*}
For the remaining terms of $T_7$, we note that $\frac{\lambda(n)}{n^\frac 13}\ll 1$ for any $n$, and hence
\begin{align*}
\sum_{m_1\leq \frac{q^{2+2\alpha}}{x}}&\mathop{\sum\sum}\limits_{d_s,m_s\leq D^\alpha}\lambda(d_s)\lambda(m_s)\mathop{\sum\sum}\limits_{l,j\leq D^2}\lambda(l)\lambda(j)x^{\kappa'}\left(\left(\frac{x}{Y}\right)^{\frac{15}{32}}+q^\frac 14\left(\frac{x}{Y}\right)^{\frac{1}{3}}\right)\\
\ll & \sum_{m_1\leq \frac{q^{2+2\alpha}}{x}}\mathop{\sum\sum}\limits_{u,v\leq D^\frac \alpha 2}\mathop{\sum\sum}\limits_{l,j\leq D^2}x^{\kappa'}\left(\left(\frac{x}{m_1}\right)^{\frac{15}{32}}+q^\frac 14\left(\frac{x}{m_1}\right)^{\frac{1}{3}}\right)\\
\ll & D^{4+\alpha}x^{\kappa'}q^{2\alpha}\left(x^{-\frac{1}{16}}q^{\frac{17}{16}}+q^\frac{19}{12}x^{-\frac 13}\right)\\
=&D^{4+\alpha}x^{\kappa'}q^{2\alpha}\frac xq\left(\left(\frac{q^\frac{33}{17}}{x}\right)^\frac{17}{16}+\left(\frac{q^{\frac{31}{16}}}{x}\right)^\frac{4}{3}\right)
\end{align*}
Since $\theta<\frac{30}{59}$, the term above is clearly
$$\ll \frac{x^{1-\rho}}{q}$$
for some positive value of $\rho$.
\end{proof}

\begin{theorem} For $Q<x^\frac{16}{31}$, there exists a $\rho>0$ such that
$$\sum_{q\sim Q}\max_{(a,q)=1}\left|T_7(a,q)\right|\ll xL(1,\chi)\log^2 x\log^4 D+x^{1-\rho}.$$
\end{theorem}
\begin{proof}
The proof is nearly identical to the previous one.  We apply Corollary \ref{KloocorIr} to find that the analogous estimate to (\ref{T7}) is now
\begin{align*}
\ll &\sum_{q\sim Q}\sum_{m_1\leq \frac{q^{2+2\alpha}}{x}}\mathop{\sum\sum}\limits_{d_s,m_s\leq D^\alpha}\lambda(d_s)\lambda(m_s)\mathop{\sum\sum}\limits_{l,j\leq D^2}\lambda(l)\lambda(j)\\
&\cdot \left[\frac{x\log xL(1,\chi)}{2qH}+x^{\kappa'}\left(Q^\frac 54\left(\frac{x}{H}\right)^\frac{5}{16}+Q\left(\frac{x}{H}\right)^\frac{9}{20}+Q^\frac 76\left(\frac{x}{H}\right)^{\frac{13}{36}}\right)\right].
\end{align*}
The sums involving $L(1,\chi)$ will now be
$$\ll xL(1,\chi)\log^2 x\log^4 D.$$
Meanwhile, the remaining sums are
\begin{align*}\ll &\sum_{m_1\leq \frac{q^{2+2\alpha}}{x}}D^{4+\alpha}x^{\kappa'}\left(Q^\frac 54\left(\frac{x}{m_1}\right)^\frac{5}{16}+Q\left(\frac{x}{m_1}\right)^\frac{9}{20}+Q^\frac 76\left(\frac{x}{m_1}\right)^{\frac{13}{36}}\right).
\end{align*}
Summing over $m_1$ then gives
\begin{align*}\ll &D^{4+\alpha}x^{\kappa'}\left(Q^\frac 54\left(\frac{Q^2}{x}\right)^\frac{11}{16}+Qx^{\frac{9}{20}}\left(\frac{Q^2}{x}\right)^\frac{11}{20}+Q^\frac 76x^{\frac{13}{36}}\left(\frac{Q^2}{x}\right)^{\frac{23}{36}}\right)\\
\ll &D^{4+\alpha}x^{1+\kappa'}\left(\left(\frac{Q^\frac{21}{11}}{x}\right)^\frac{11}{8}+\left(\frac{Q^\frac{21}{11}}{x}\right)^\frac{11}{10} +\left(\frac{Q^\frac{44}{23}}{x}\right)^\frac{23}{18}\right).
\end{align*}
Since $Q<x^\frac{16}{31}$ and $\frac{16}{31}<\frac{23}{44}<\frac{11}{21}$, the above is $\ll x^{1-\rho}$ for some $\rho$.

\end{proof}

\section{The Sum $T_1$: Prelude and Edge Cases}\label{binsumsec}

The last step will be to prove the bound on $T_1$.  This is similar to Section 5 in \cite{FI03}, wherein the authors use results on ternary sums in arithmetic progressions to find a bound.  Here, we show how one can improve on these methods by proceeding more carefully in the treatment of these sums and by using a new result of Shparlinski \cite{Shp}.

As noted above, the primary constraint in improving the level of distribution is the equality
$$\sum_{\substack{dm_1\leq x \\ d>D^* \\ dm_1\equiv a\pmod q}}\lambda(d)=(1+o(1))\frac{1}{\phi(q)}\sum_{\substack{dm_1\leq x \\ d>D^* \\ (dm_1, q)=1}}\lambda(d).$$
Thus, we turn now to the theory of binary and ternary divisor sums.

We can write
$$\sum_{\substack{dm_1\leq x \\ d>D^* \\ dm_1\equiv a\pmod q}}\lambda(d)=\mathop{\sum\sum\sum}\limits_{\substack{d_1d_2m_1\leq x \\ d_1d_2>D^* \\ d_1d_2m_1\equiv a\pmod q}}\chi(d_1).$$
We will eventually split this into dyadic intervals.  To do this, however, it will help to eliminate the cases where two of the variables are very small and the third can be very large.  To this end, we first prove the following.

\begin{lemma}\label{easy1} For $\alpha$ as before and $q<x^{\frac 23-\alpha}$,
$$\mathop{\sum\sum}\limits_{\substack{d_1,d_2 \\ D^*<d_1d_2\leq \frac{x}{Dq^{1+\alpha}}}}\chi(d_1)\sum_{\substack{Dq^{1+\alpha}\leq m_1\leq \frac{x}{d_1d_2}\\m_1\equiv a\overline{d_1d_2}\pmod q}}1=\frac 1{\phi(q)}\mathop{\sum\sum}\limits_{\substack{d_1,d_2\\ D^*<d_1d_2\leq \frac{x}{Dq^{1+\alpha}}}}\chi(d_1)\sum_{\substack{Dq^{1+\alpha}\leq m_1\leq \frac{x}{d_1d_2}\\(d_1d_2m_1,q)=1}}1+O\left(\frac{x}{Dq^{1+\alpha}}\right),$$
and
$$\mathop{\sum\sum}\limits_{\substack{d_1,m_1\\ d_1m_1\leq \frac{x}{Dq^{1+\alpha}}}}\chi(d_1)\sum_{\substack{Dq^{1+\alpha}\leq d_2\leq \frac{x}{d_1m_1}\\d_2\equiv a\overline{d_1m_1}\pmod q}}1=\frac 1{\phi(q)}\mathop{\sum\sum}\limits_{\substack{d_1,m_1\\ d_1m_1\leq \frac{x}{Dq^{1+\alpha}}}}\chi(d_1)\sum_{\substack{Dq^{1+\alpha}\leq d_2\leq \frac{x}{d_1m_1}\\(d_1d_2m_1,q)=1}}1+O\left(\frac{x}{Dq^{1+\alpha}}\right).$$

\end{lemma}
\begin{proof}
We prove the first equality; the second is the same.  For $d_1, d_2$ with $(d_1d_2,q)=1$, the sum over $m_1$ can be rewritten as
$$\sum_{\substack{Dq^{1+\alpha}\leq m_1\leq \frac{x}{d_1d_2}\\m_1\equiv a\overline{d_1d_2}\pmod q}}1=\frac 1{\phi(q)}\sum_{\substack{Dq^{1+\alpha}\leq m_1\leq \frac{x}{d_1d_2}\\(m_1,q)=1}}1+O(1).$$
Thus
\begin{align*}\mathop{\sum\sum}\limits_{\substack{d_1,d_2\\ d_1d_2\leq \frac{x}{Dq^{1+\alpha}}}}\chi(d_1)\sum_{\substack{Dq^{1+\alpha}\leq m_1\leq \frac{x}{d_1d_2}\\m_1\equiv a\overline{d_1d_2}\pmod q}}1=&\mathop{\sum\sum}\limits_{\substack{d_1,d_2\\ d_1d_2\leq \frac{x}{Dq^{1+\alpha}}\\ (d_1d_2,q)=1}}\chi(d_1)\left(\frac 1{\phi(q)}\sum_{\substack{Dq^{1+\alpha}\leq m_1\leq \frac{x}{d_1d_2}\\(m_1,q)=1}}1+O\left(1\right)\right)\\
=&\frac 1{\phi(q)}\mathop{\sum\sum}\limits_{\substack{d_1,d_2\\ d_1d_2\leq \frac{x}{Dq^{1+\alpha}}}}\chi(d_1)\sum_{\substack{Dq^{1+\alpha}\leq m_1\leq \frac{x}{d_1d_2}\\(d_1d_2m_1,q)=1}}1+O\left(\frac{x}{Dq^{1+\alpha}}\right).
\end{align*}
\end{proof}

For the next important sum (and for future reference), it will be helpful to know that relative primality conditions will have little impact on our sums of characters.  To this end, we recall the identity that
\begin{gather}\label{chiId}
\sum_{n\leq y}\frac{\chi(n)}{n}=L(1,\chi)+O\left(\frac Dy\right).
\end{gather}
This can be easily verified by noting that
$$\sum_{n\leq y}\frac{\chi(n)}{n}=L(1,\chi)-\lim_{z\to \infty}\sum_{y<n\leq z}\frac{\chi(n)}{n},$$
and then applying partial summation to the latter sum.

With this identity, we prove the following two lemmas:
\begin{lemma}\label{chiId2} Let $0<\kappa<\frac 14$ and $A>1$.  Moreover, let $D$, $s$, and $y$ be natural numbers where $y^\kappa<s<y^A$, $D<y^{1-2\kappa}$, and $(D,s)=1$.  Then for any $\varepsilon>0$,
$$\sum_{\substack{n\leq y \\ (n,s)=1}}\chi(n)\ll Ds^{\varepsilon},$$
and
$$\sum_{\substack{n\leq y \\ (n,s)=1}}\frac{\chi(n)}{n}\ll L(1,\chi)\log y.$$
\end{lemma}

\begin{proof}
As usual, we let $\omega(n)$ denote as usual the number of distinct prime factors of $n$.  Then by inclusion/exclusion,
$$\sum_{\substack{n\leq y \\ (n,s)=1}}\chi(n)=\sum_{r|s}(-1)^{\omega(r)}\mu(r)^2\sum_{n'\leq \frac yr}\chi(r)\chi(n')\ll D\tau(s)\ll Ds^{\varepsilon}.$$
For the latter expression,
\begin{align*}
\sum_{\substack{n\leq y \\ (n,s)=1}}\frac{\chi(n)}{n}=&\sum_{r|s}(-1)^{\omega(r)}\mu(r)^2\sum_{n'\leq \frac yr}\frac{\chi(r)}{r}\frac{\chi(n')}{n'}\\
=&\sum_{\substack{r|s \\ r\leq y^{\frac \kappa 2}}}\frac{(-1)^{\omega(r)}\mu(r)^2\chi(r)}{r}\sum_{n'\leq \frac yr}\frac{\chi(n')}{n'}+\sum_{\substack{r|s \\ r>y^{\frac \kappa 2}}}\frac{(-1)^{\omega(r)}\mu(r)^2\chi(r)}{r}\sum_{n'\leq \frac yr}\frac{\chi(n')}{n'}.\end{align*}
By (\ref{chiId}),
\begin{align*}
\sum_{\substack{r|s \\ r\leq y^{\frac \kappa 2}}}\frac{(-1)^{\omega(r)}\mu(r)^2\chi(r)}{r}\sum_{n'\leq \frac yr}\frac{\chi(n')}{n'}=&\sum_{\substack{r|s \\ r\leq y^{\frac \kappa 2}}}\frac{(-1)^{\omega(r)}\mu(r)^2\chi(r)}{r}\left(L(1,\chi)+O\left(\frac{Dr}{y}\right)\right)\\
\ll &L(1,\chi)\log y+\frac{Ds^\varepsilon\log y}{y}\end{align*}
for any $\varepsilon>0$.  Meanwhile,
$$\sum_{\substack{r|s \\ r>y^{\frac \kappa 2}}}\frac{(-1)^{\omega(r)}\mu(r)^2\chi(r)}{r}\sum_{n'\leq \frac yr}\frac{\chi(n')}{n'}\ll \frac{\tau(s)}{y^{\frac \kappa 2}}\sum_{n'\leq y}\frac{1}{n}\ll \frac{\log y}{y^{\frac \kappa 3}}.$$
The $L(1,\chi)\log y$ term is clearly the largest, and hence the lemma follows.
\end{proof}
\begin{lemma} \label{chiId3}Let $\kappa$, $A$, $s$, $D$, and $y$ be as in the previous lemma, and let $z>y$.  Then for any $\varepsilon>0$,
$$\sum_{\substack{y<n\leq z \\ (n,s)=1}}\frac{\chi(n)}{n}\ll \frac{Ds^\varepsilon}{y}.$$
\end{lemma}
\begin{proof}
Proceeding as in the last lemma,
\begin{align*}
\sum_{\substack{y<n\leq z \\ (n,s)=1}}\frac{\chi(n)}{n}=&\sum_{r|s}\frac{(-1)^{\omega(r)}\mu(r)^2\chi(r)}{r}\sum_{\frac yr<n'\leq \frac zr}\frac{\chi(n')}{n'}\\
\ll &\sum_{r|s}\frac{1}{r}\left(\frac{Dr}{y}\right)\ll \frac{D\tau(s)}{y}\ll \frac{Ds^\varepsilon}{y}.
\end{align*}

\end{proof}

In many of the upcoming lemmas, we will work mod $Dq$ instead of $q$, since this will make it easier to handle $\chi$.  We discuss the reasons for this change a bit more the next section.  However, it is easy to convert these back to statements about mod $q$, as we will see below.
\begin{lemma}\label{easy2} Let $\sqrt x\leq q<x^{\frac 23-\alpha}$.  Then
\begin{gather}\label{bigd0}\mathop{\sum\sum}\limits_{\substack{d_2,m_1\\ d_2m_1\leq \frac{x}{Dq^{1+\alpha}}}}\sum_{\substack{Dq^{1+\alpha}\leq d_1\leq\frac{x}{d_2m_1}\\d_1\equiv a\overline{d_2m_1}\pmod{Dq}}}\chi(d_1)\ll \frac{1}{Dq}L(1,\chi)^2\log^2 x+\frac{x}{Dq^{1+\alpha}}.\end{gather}
Consequently,
\begin{gather}\label{bigd1}\mathop{\sum\sum}\limits_{\substack{d_2,m_1\\ d_2m_1\leq \frac{x}{Dq^{1+\alpha}}}}\sum_{\substack{Dq^{1+\alpha}\leq d_1\leq \frac{x}{d_2m_1}\\d_1\equiv a\overline{d_2m_1}\pmod{q}}}\chi(d_1)=\frac{1}{\phi(q)}\mathop{\sum\sum}\limits_{\substack{d_2,m_1\\ d_2m_1\leq \frac{x}{Dq^{1+\alpha}}}}\sum_{\substack{Dq^{1+\alpha}\leq d_1\leq\frac{x}{d_2m_1}\\(d_1d_2m_1,q)=1}}\chi(d_1)+O\left(\frac{x}{q}L(1,\chi)^2\log^2x \right).\end{gather}
\end{lemma}

\begin{proof}
Beginning with the first claim, the congruence gives $d_1\equiv a\overline{d_2m_1}\pmod{D}$.  Additionally, $\overline{\chi(n)}=\chi(n)$, since $\chi$ is real.  So $\chi(d_1)=\chi(ad_2m_1)$.  Plugging this in:
\begin{align*}
\mathop{\sum\sum}\limits_{\substack{d_2,m_1\\ d_2m_1\leq \frac{x}{Dq^{1+\alpha}}}}&\sum_{\substack{Dq^{1+\alpha}\leq d_1\leq\frac{x}{d_2m_1}\\d_1\equiv a\overline{d_2m_1}\pmod{Dq}}}\chi(d_1)=\chi(a)\mathop{\sum\sum}\limits_{\substack{d_2,m_1\\ d_2m_1\leq \frac{x}{Dq^{1+\alpha}}}}\chi(m_1)\chi(d_2)\sum_{\substack{Dq^{1+\alpha}\leq d_1\leq\frac{x}{d_2m_1}\\d_1\equiv a\overline{d_2m_1}\pmod{Dq}}}1.
\end{align*}
We can then sum over $d_1$.  Note that the presence of the character now means that we do not have to be explicit about the relative primality between $d_2m_1$ and $D$, since $\chi(d_2m_1)=0$ if they are not relatively prime.  So the above is
\begin{align*}
=&\chi(a)\mathop{\sum\sum}\limits_{\substack{d_2,m_1\\ d_2m_1\leq \frac{x}{Dq^{1+\alpha}}\\(d_2m_1,q)=1}}\chi(m_1)\chi(d_2)\left( \frac{x}{Dqd_2m_1}+O(1)\right)\\
=&\frac{x\chi(a)}{Dq}\mathop{\sum\sum}\limits_{\substack{d_2,m_1\\ d_2m_1\leq \frac{x}{Dq^{1+\alpha}}\\(d_2m_1,q)=1}}\frac{\chi(m_1)}{m_1}\frac{\chi(d_2)}{d_2}+O\left(\frac{x}{Dq^{1+\alpha}}\right).
\end{align*}
By Dirchlet's hyperbola trick and the symmetry of the sum,
\begin{align*}
\mathop{\sum\sum}\limits_{\substack{d_2,m_1\\ d_2m_1\leq \frac{x}{Dq^{1+\alpha}}\\(d_2m_1,q)=1}}\frac{\chi(d_2)}{d_2}\frac{\chi(m_1)}{m_1}=&2\sum_{\substack{d_2\leq \sqrt{\frac{x}{Dq^{1+\alpha}}}\\(d_2,q)=1}}\frac{\chi(d_2)}{d_2}\sum_{\substack{m_1\leq \frac{x}{Dq^{1+\alpha}d_2}\\(m_1,q)=1}}\frac{\chi(m_1)}{m_1}-\left(\sum_{\substack{d_2\leq \sqrt{\frac{x}{Dq^{1+\alpha}}}\\(d_2,q)=1}}\frac{\chi(d_2)}{d_2}\right)^2
\end{align*}
We can increase the bound for $m_1$, thereby removing the dependence on $d_2$ on the cutoff for the sum over $m_1$.  By Lemmas \ref{chiId2} and \ref{chiId3}, for any $\varepsilon'>0$ the above is
\begin{align*}
=&2\sum_{\substack{d_2\leq \sqrt{\frac{x}{Dq^{1+\alpha}}}\\(d_2,q)=1}}\frac{\chi(d_2)}{d_2}\left(\sum_{\substack{m_1\leq x^{100}\\(m_1,q)=1}}\frac{\chi(m_1)}{m_1}+O\left(\frac{D^2q^{1+\alpha+\varepsilon'}d_2}{x}\right)\right)-\left(\sum_{\substack{d_2\leq \sqrt{\frac{x}{Dq^{1+\alpha}}}\\(d_2,q)=1}}\frac{\chi(d_2)}{d_2}\right)^2\\
=&2\sum_{\substack{d_2\leq \sqrt{\frac{x}{Dq^{1+\alpha}}}\\(d_2,q)=1}}\frac{\chi(d_2)}{d_2}\sum_{\substack{m_1\leq x^{100}\\(m_1,q)=1}}\frac{\chi(m_1)}{m_1}-\left(\sum_{\substack{d_2\leq \sqrt{\frac{x}{Dq^{1+\alpha}}}\\(d_2,q)=1}}\frac{\chi(d_2)}{d_2}\right)^2+O\left(\frac{D^\frac 32q^{\frac 12+\frac{\alpha}{2}+\varepsilon'}}{x^\frac 12}\right)\\
\ll &L(1,\chi)^2\log^2 x+\frac{D^\frac 32q^{\frac 12+\frac{\alpha}{2}+\varepsilon'}}{x^\frac 12},
\end{align*}
So
\begin{align*}
\mathop{\sum\sum}\limits_{\substack{d_2,m_1\\ d_2m_1\leq \frac{x}{Dq^{1+\alpha}}\\(d_2m_1,q)=1}}&\sum_{\substack{Dq^{1+\alpha}\leq d_1\leq\frac{x}{d_2m_1}\\d_1\equiv a\overline{d_2m_1}\pmod{Dq}}}\chi(d_1)\ll \frac{x}{Dq}L(1,\chi)^2\log^2 x+\frac{x^\frac 12D^\frac 12}{q^{\frac 12-\frac{\alpha}{2}-\varepsilon'}}+\frac{x}{Dq^{1+\alpha}}.
\end{align*}
The latter expression in the error term is larger than the middle term as long as $q^{1+3\alpha+\varepsilon'}D^3\leq x$, which is clearly true by our assumption on $q$ as long as $\varepsilon'$ is small (say, $\varepsilon'<\frac \alpha 2$).

To show that (\ref{bigd1}) does indeed follow from (\ref{bigd0}), note that trivially, Lemma \ref{chiId2} gives
$$\frac{1}{\phi(q)}\mathop{\sum\sum}\limits_{\substack{d_2,m_1\\ d_2m_1\leq \frac{x}{Dq^{1+\alpha}}}}\sum_{\substack{Dq^{1+\alpha}\leq d_1\leq\frac{x}{d_2m_1} \\ (d_1d_2m_1,Dq)=1}}\chi(d_1)\ll \frac{1}{\phi(q)}\mathop{\sum\sum}\limits_{\substack{d_2,m_1\\ d_2m_1\leq \frac{x}{Dq^{1+\alpha}}\\(d_2m_1,q)=1}}Dq^{\frac{\alpha}{2}}\ll \frac{x}{q^{1+\frac \alpha 2}\phi(q)}.$$
Also, we can change the congruence in (\ref{bigd0}) from $Dq$ back to $q$:
\begin{align*}\mathop{\sum\sum}\limits_{\substack{d_2,m_1\\ d_2m_1\leq \frac{x}{Dq^{1+\alpha}}}}\sum_{\substack{Dq^{1+\alpha}\leq d_1\leq\frac{x}{d_2m_1}\\d_1\equiv a\overline{d_2m_1}\pmod{q}}}\chi(d_1)=&\sum_{\substack{0<a'<D \\ \phantom{1}D\nmid a+a'q\phantom{1}}}\mathop{\sum\sum}\limits_{\substack{d_2,m_1\\ d_2m_1\leq \frac{x}{Dq^{1+\alpha}}}}\sum_{\substack{Dq^{1+\alpha}\leq d_1\leq\frac{x}{d_2m_1}\\d_1\equiv (a+a'q)\overline{d_2m_1}\pmod{Dq}}}\chi(d_1)\\
\ll &\frac{x}{q}L(1,\chi)^2\log^2 x+\frac{x}{q^{1+\alpha}}.
\end{align*}
Since all of these terms are bounded by the error term, the second half of the lemma vacuously follows.
\end{proof}

\section{Binary Sums and $\lambda(n)$}\label{binsumsec1a}

In this section, we prove some results for binary divisor sums that are twisted by a character.  These will apply to the cases where only two of the variables are large enough to affect our computations.  In other words, for some bound $\mathcal B$, we will prove lemmas that will help us with sums of the form
$$\sum_{d_1<\mathcal B}\chi(d_1)\mathop{\sum\sum}\limits_{\substack{d_2m \leq \frac{x}{d_1}\\ d_1d_2m_1\equiv a\pmod q}}1$$
and
$$\sum_{d_2<\mathcal B}\mathop{\sum\sum}\limits_{\substack{d_1m \leq \frac{x}{d_2}\\ d_1d_2m_1\equiv a\pmod q}}\chi(d_1).$$
In these cases, we can apply the congruence condition to the inside double sum and treat the outside sum later.


We will first prove our binary result dyadically or sub-dyadically, and then we prove it in generality.  The proofs here will follow very closely the standard proof of Theorem \ref{binarydivisor}, as we are simply modifying this proof to show that it still holds with the inclusion of our character mod $D$.

We note that these lemmas and corollaries are much easier if one assumes that $q$ and $D$ are relatively prime.  (Indeed, Friedlander and Iwaniec largely do this for the proofs in \cite{FI03}.)  This is because the variable $n=h+n'q$ will be equidistributed mod $D$ as $n'$ ranges over an interval $I$, which means that a sum $\sum_{n'\in I}\chi(h+n'q)$ will be bounded by $D$ regardless of the choice of $I$.  However, the proof becomes much trickier when $1<(q,D)<D$, because then $h+n'q$ does not range over all classes, and hence a sum over $n'$ may not have much cancellation.  The easiest solution, then, is to work mod $qD$, in which case $\chi(h+n'qD)=\chi(h)$ can be treated as a constant with regard to $n'$.  This is indeed what we do for much of the rest of the paper.

\begin{lemma}\label{Weildyad}
For $x$ and $q$ with $\sqrt x\leq q\leq \frac{x^{\frac 23-\alpha}}{D^\frac 32}$, let $1<\zeta<2$, $UV\leq x$, and $U,V\geq q^{\kappa}$ for some constant $\kappa$ with $\kappa>4\alpha+4\eta$.  Define $\mathcal U$ and $\mathcal V$ to be the intervals $\mathcal U=[U,\zeta U)$, $\mathcal V=[V,\zeta V)$.  Then for any $a$ with $(a,Dq)=1$,
$$\mathop{\sum\sum}\limits_{\substack{u\in \mathcal U,v\in \mathcal V\\ uv\equiv a\pmod{Dq}}}\chi(v)=\frac{1}{Dq}\mathop{\sum\sum}\limits_{\substack{u\in \mathcal U,v\in \mathcal V\\ (u,Dq)=1}}\chi(v)+\frac{\chi(a)}{Dq}\mathop{\sum\sum}\limits_{\substack{u\in \mathcal U,v\in \mathcal V\\ (u,Dq)=1}}\chi(u)+O\left(\frac{x(\zeta-1)}{q^{1+\kappa-\alpha-2\eta}}+D^\frac 32q^{\frac 12+\frac \alpha 3}\right).$$
\end{lemma}
Recall that $\eta$ was defined such that $D<x^\eta$.  

Before we prove this, it is important to note that these sums are generally small.  Indeed, we see that
\begin{align*}
\left|\frac{1}{Dq}\mathop{\sum\sum}\limits_{\substack{u\in \mathcal U,v\in \mathcal V\\ (u,Dq)=1}}\chi(v)\right|\ll &\frac{U(\zeta-1)}{q} \ll \frac{x(\zeta-1)}{q^{1+\kappa}},\end{align*}
and by Lemma \ref{chiId2},
\begin{align*}
\left|\frac{1}{Dq}\mathop{\sum\sum}\limits_{\substack{u\in \mathcal U,v\in \mathcal V\\ (u,Dq)=1}}\chi(u)\right|\ll &\frac{V(\zeta-1)}{q^{1-\alpha}} \ll \frac{x(\zeta-1)}{q^{1+\kappa-\alpha}}.  \end{align*}

So we can rewrite the result in this lemma as
\begin{gather}\label{dyadrewr}\mathop{\sum\sum}\limits_{\substack{u\in \mathcal U,v\in \mathcal V\\ uv\equiv a\pmod{Dq}}}\chi(v)=\frac{1}{\phi(Dq)}\mathop{\sum\sum}\limits_{\substack{u\in \mathcal U,v\in \mathcal V\\ (uv,q)=1}}\chi(v)+O\left(\frac{x(\zeta-1)}{q^{1+\kappa-\alpha-2\eta}}+D^\frac 32q^{\frac 12+\frac \alpha 3}\right),
\end{gather}
or simply
\begin{gather}\label{dyadrewr2}\left|\mathop{\sum\sum}\limits_{\substack{u\in \mathcal U,v\in \mathcal V\\ uv\equiv a\pmod{Dq}}}\chi(v)\right|\ll \frac{x(\zeta-1)}{q^{1+\kappa-\alpha-2\eta}}+D^\frac 32q^{\frac 12+\frac \alpha 3}.
\end{gather}

\begin{proof}
For ease of notation, define
$$N(U,V,q,a,\zeta)=\mathop{\sum\sum}\limits_{\substack{u\in \mathcal U,v\in \mathcal V\\ uv\equiv a\pmod{Dq}}}\chi(v).$$
Note that we can separate out the character via a congruence condition on $v$:
\begin{align}\label{hsplit1}
N(U,V,q,a,\zeta)=\sum_{h=1}^D\chi(h)\mathop{\sum\sum}\limits_{\substack{u\in \mathcal U,v\in \mathcal V\\ uv\equiv a\pmod{Dq}\\ v\equiv h\pmod D}}1\end{align}
For ease of notation, we will write $\mathcal U_{Dq}$ to indicate the elements of $\mathcal U$ that are relatively prime to $Dq$.  Clearly, if $uv\equiv a\pmod{Dq}$ then $u$ must be in $\mathcal U_{Dq}$.  We can then express $N$ with exponential sums:
\begin{align*}
N(U,V,q,a,\zeta)=&\frac{1}{D^2q}\sum_{h=1}^D\chi(h)\sum_{u\in \mathcal U_{Dq}}\sum_{v\in \mathcal V}\sum_{s=1}^{Dq}\sum_{r=1}^De_D\left(r v-rh\right)e_{Dq}\left(s\bar ua-sv\right)\\
=&\frac{1}{D^2q}\sum_{h=1}^D\chi(h)\sum_{s=1}^{Dq}\sum_{r=1}^{D}e_{D}\left(-rh\right)\sum_{u\in \mathcal U_{Dq}}\sum_{v\in \mathcal V}e_{Dq}\left(as\bar u\right)e_{Dq}\left((rq-s)v\right)
\end{align*}
Let us express by $N_1$ the restriction of this sum to $s=qD$.  So
\begin{align*}
N_1=&\frac{1}{D^2q}\sum_{h=1}^D\chi(h)\sum_{r=1}^D\sum_{u\in \mathcal U_{Dq}}\sum_{v\in \mathcal V}e_{D}\left(rv-rh\right)\\
&=\frac{1}{Dq}\sum_{h=1}^D\chi(h)\sum_{\stackrel{v\in \mathcal V}{v\equiv h\pmod D}}\sum_{u\in \mathcal U_{Dq}}1\\
&=\frac{1}{Dq}\sum_{v\in \mathcal V}\sum_{u\in \mathcal U_{Dq}}\chi(v).
\end{align*}
This comprises the first part of the required main term.

Next, let $N_2$ denote the restriction of $N$ to the cases where $rq-s\equiv 0\pmod{Dq}$.  Clearly, this must mean that $q|s$.  So we let $s=qs'$, which means that
$$r-s'\equiv 0\pmod{D}.$$
Since $s'\leq D$, this means that $r=s'$.  Thus
\begin{align*}
N_2=&\frac{1}{D^2q}\sum_{h=1}^D\chi(h)\sum_{r=1}^{D}e_{D}\left(-rh\right)\sum_{u\in \mathcal U_{Dq}}\sum_{v\in \mathcal V}e_{Dq}\left(arq\bar u\right)\\
=&\frac{1}{Dq}\sum_{h=1}^D\chi(h)\sum_{\substack{u\in \mathcal U_{Dq}\\ u\equiv a\bar h\pmod D}}\sum_{v\in \mathcal V}1\\
=&\frac{\chi(a)}{Dq}\sum_{u\in \mathcal U_{q}}\sum_{v\in \mathcal V}\chi(u).
\end{align*}
Finally, we note that $N_1$ and $N_2$ both count the case where $r=D$ and $s=qD$.  Letting $N_3$ denote this double-counted case, we have
$$N_3=\frac{1}{D^2q}\sum_{h=1}^D\chi(h)\sum_{u\in \mathcal U_{Dq}}\sum_{v\in \mathcal V}1=0,$$
by the orthogonality of the character.

So we then consider $N_0=N-N_1-N_2+N_3$.  We will show that $N_0$ is small, thereby proving the lemma.  Here,
\begin{align*}
N_0=&\frac{1}{D^2q}\sum_{h=1}^D\chi(h)\mathop{\sum\sum}\limits_{\substack{1\leq s\leq qD-1\\1\leq r\leq D \\ rq-s\not \equiv 0\pmod{Dq}}}e_{D}\left(-rh\right)\sum_{u\in \mathcal U_{Dq}}\sum_{v\in \mathcal V}e_{Dq}\left(as\bar u\right)e_{Dq}\left((rq-s)v\right)\\
=&\frac{1}{D^2q}\sum_{h=1}^D\chi(h)\sum_{r=1}^{D}\sum_{\substack{d|qD \\ q\nmid d}}\sum_{\substack{1\leq k\leq \frac{Dq}{d} \\ \left(k,\frac{Dq}{d}\right)=1}}e_{D}\left(-rh\right)\sum_{u\in \mathcal U_{Dq}}\sum_{v\in \mathcal V}e_{qD/d}\left(ak\bar u\right)e_{Dq}\left((rq-kd)v\right)\\
&+\frac{1}{D^2q}\sum_{h=1}^D\chi(h)\sum_{r=1}^{D}\sum_{\substack{d'|D \\ d'<D }}\sum_{\substack{1\leq k\leq \frac{D}{d'} \\ \left(k,\frac D{d'}\right)=1 \\ r\neq kd'}}e_{D}\left(-rh\right)\sum_{u\in \mathcal U_{Dq}}\sum_{v\in \mathcal V}e_{D}\left(akd'\bar u\right)e_{D}\left((r-kd')v\right),
\end{align*}
where $s=dk$ or $s=qd'k$.

Examining the latter sum, we have
\begin{align*}
\frac{1}{D^2q}&\left|\sum_{h=1}^D\chi(h)\sum_{r=1}^{D}\sum_{\substack{d'|D \\ d'<D }}\sum_{\substack{1\leq k\leq \frac{D}{d'} \\ \left(k,\frac D{d'}\right)=1 \\ r\neq kd'}}e_{D}\left(rh\right)\sum_{u\in \mathcal U_{Dq}}\sum_{v\in \mathcal V}e_{D}\left(akd'\bar u\right)e_{D}\left((r-kd')v\right)\right|\\
\ll &\frac{1}{D^2q}\sum_{h=1}^D\sum_{r=1}^{D}\sum_{\substack{d'|D \\ d'<D }}\sum_{\substack{1\leq k\leq \frac{D}{d'} \\ \left(k,\frac D{d'}\right)=1 \\ r\neq kd'}}\left|\sum_{u\in \mathcal U_{Dq}}e_{D/d'}\left(ak\bar u\right)\right|\frac{1}{||(r-kd')/D||}
\end{align*}
We can evaluate the sum over $u$ trivially.  Letting $r'\equiv r-kd'$ mod $D$, the above is then
\begin{align}\label{error1}\ll \frac{U(\zeta-1)}{D^2q}&\sum_{h=1}^D\sum_{\substack{d'|D \\ d'<D }}\sum_{\substack{1\leq k\leq \frac{D}{d'} \\ \left(k,\frac D{d'}\right)=1 }}\sum_{r'=1}^{D-1}\frac{1}{||r'/D||}\ll \frac{DU(\zeta-1)\log^2 x}{q}\ll \frac{x(\zeta-1)}{q^{1+\kappa-\alpha-2\eta}}.
\end{align}

For the former, write $d=t_1t_2$, where $t_1|q$ and $(\frac{q}{t_1},t_2)=1$.  So
\begin{align*}
\frac{1}{D^2q}&\left|\sum_{h=1}^D\chi(h)\sum_{r=1}^{D}\sum_{\substack{d|qD \\ q\nmid d}}\sum_{\substack{1\leq k\leq \frac{Dq}{d} \\ \left(k,\frac{Dq}{d}\right)=1}}e_{D}\left(rh\right)\sum_{u\in \mathcal U_{Dq}}\sum_{v\in \mathcal V}e_{qD/d}\left(ak\bar u\right)e_{Dq}\left((rq-kd)v\right)\right|\\
\ll &\frac{1}{Dq}\sum_{r=1}^{D}\sum_{\substack{t_1|q \\ t_1<q}}\sum_{\substack{t_2|D \\ (t_2,\frac{q}{t_1})=1 }}\sum_{\substack{1\leq k\leq \frac{Dq}{t_1t_2} \\ \left(k,\frac{Dq}{t_1t_2}\right)=1}}\left|\sum_{u\in \mathcal U_{Dq}}e_{qD/t_1t_2}\left(ak\bar u\right)\right|\left|\sum_{v\in \mathcal V}e_{qD/t_1}\left(\left(r\frac{q}{t_1}-kt_2\right)v\right)\right|\\
\ll &\frac{1}{Dq}\sum_{r=1}^{D}\sum_{\substack{t_1|q \\ t_1<q}}\sum_{\substack{t_2|D \\ (t_2,\frac{q}{t_1})=1 }}\sum_{\substack{1\leq k\leq \frac{Dq}{t_1t_2} \\ \left(k,\frac{Dq}{t_1t_2}\right)=1}}\left|\sum_{u\in \mathcal U_{Dq}}e_{qD/t_1t_2}\left(ak\bar u\right)\right|\frac{1}{||(r\frac{q}{t_1}-kt_2)/(qD/t_1)||}
\end{align*}
In order to apply Kloosterman results to the exponential sum modulo $qD/d$, we must loosen the restriction that $(u,Dq)=1$ to a restriction that $\left(u,\frac{Dq}d\right)=1$.  To this end, we can use the fact that
\begin{align*}\sum_{\stackrel{u\in \mathcal U}{(u,Dq)=1}}e_{Dq}(ak\bar u)=&\sum_{\stackrel{u\in \mathcal U}{(u,Dq/d)=1}}e_{Dq/d}(ak\bar u)\sum_{l|(d,u)}\mu(l)\\
=&\sum_{l|d}\mu(l)\sum_{\stackrel{lj\in \mathcal U}{(j,Dq/d)=1}}e_{Dq/d}(ak\overline{lj}).
\end{align*}
We use this to bound the quintuple sum above:
\begin{align*}
\ll &\frac{1}{Dq}\sum_{r=1}^{D}\sum_{\substack{t_1|q \\ t_1<q}}\sum_{\substack{t_2|D \\ (t_2,\frac{q}{t_1})=1 }}\sum_{\substack{1\leq k\leq \frac{Dq}{t_1t_2} \\ \left(k,\frac{Dq}{t_1t_2}\right)=1}}\sum_{l|t_1t_2}\left|\sum_{\substack{(j,qD/t_1t_2)=1 \\ lj\in \mathcal U}}e_{qD/t_1t_2}\left(ak\overline{lj}\right)\right|\frac{1}{||(r\frac{q}{t_1}-kt_2)/(qD/t_1)||}.
\end{align*}
We can simplify the remaining sum by evaluating the parts that are not the exponential sum.  Letting $k'=r\frac{q}{t_1}-kt_2$, we have
\begin{align*}
\sum_{r=1}^{D}&\sum_{\substack{t_1|q \\ t_1<q}}\sum_{\substack{t_2|D \\ (t_2,\frac{q}{t_1})=1 }}\sum_{\substack{1\leq k\leq \frac{Dq}{t_1t_2} \\ \left(k,\frac{Dq}{t_1t_2}\right)=1}}\frac{1}{||(r\frac{q}{t_1}-kt_2)/(qD/t_1)||}\\
\ll &\sum_{r=1}^{D}\sum_{\substack{t_1|q \\ t_1<q}}\sum_{\substack{t_2|D \\ (t_2,\frac{q}{t_1})=1 }}\sum_{1\leq k'<\frac{Dq}{t_1} }\frac{1}{||k'/(qD/t_1)||}\\
\ll &D^2q\log(Dq)\sum_{\substack{t_1|q \\ t_1<q}}\sum_{\substack{t_2|D \\ (t_2,\frac{q}{t_1})=1 }}\frac{1}{t_1}
\end{align*}
So we bound the sextuple sum:
\begin{align*}
\ll &D\log(Dq)\sum_{\substack{t_1|q \\ t_1<q}}\sum_{\substack{t_2|D \\ (t_2,\frac{q}{t_1})=1 }}\sum_{l|t_1t_2}\frac{1}{t_1}\max_{(c,qD/t_1t_2)=1}\left|\sum_{\substack{(j,qD/t_1t_2)=1 \\ lj\in \mathcal U}}e_{qD/t_1t_2}\left(c\overline{j}\right)\right|.
\end{align*}
Recall that the Weil bound for Kloosterman sums states that for any $c$ with $(c,q/d)=1$,
$$\left|\sum_{\stackrel{lj\in \mathcal U}{(j,q/d)=1}}e_{Dq/d}\left(c\bar j\right)\right|\ll \frac{U(\zeta-1) d}{Dql}+\left(\frac{D q}{d}\right)^{\frac 12+\varepsilon'}$$
for any $\varepsilon'>0$.  Applying this to bound our remaining quadruple sum:
\begin{align*}
\ll &D\log(Dq)\sum_{\substack{t_1|q \\ t_1<q}}\sum_{\substack{t_2|D \\ (t_2,\frac{q}{t_1})=1 }}\sum_{l|t_1t_2}\left(\frac{U(\zeta-1) t_2}{Dql}+\frac{1}{t_1}\left(\frac{D q}{t_1t_2}\right)^{\frac 12+\varepsilon'}\right)\\
\ll &Dq^{2\varepsilon'}\left(\frac{U(\zeta-1)}{q}+\left(D q\right)^{\frac 12}\right)\\
\ll &\frac{x(\zeta-1)}{q^{1+\kappa-2\varepsilon'-2\eta}}+D^\frac 32 q^{\frac 12+2\varepsilon'}.
\end{align*}
Taking $\varepsilon'=\frac \alpha 6$ and combining this with (\ref{error1}) then gives the theorem.

\end{proof}

\section{Binary Sums, Part 2}\label{Binsumsec2}

We can now apply these lemmas and corollaries to the ternary sum in $T_1$.  Recall that we have already handled the case where one of the variables is $\geq Dq^{1+\alpha}$.

\begin{lemma}\label{binarylemma} Let $\mathcal B=\frac{x}{D^\frac 52q^{\frac 32+\alpha}}$.  Assume that $\sqrt x\leq q\leq \frac{x^{\frac 23-2\alpha}}{D}$. Let $\rho$ be such that $0<\rho\leq \frac \alpha 8$.  Then
$$\sum_{d_1<\mathcal B}\chi(d_1)\mathop{\sum\sum}\limits_{\substack{d_2m_1\leq \frac{x}{d_1}\\d_2,m_1<q^{1+\alpha}D\\ d_1d_2m_1\equiv a\pmod{Dq}}}1=\frac{1}{\phi(Dq)}\sum_{d_1<\mathcal B}\chi(d_1)\mathop{\sum\sum}\limits_{\substack{d_2m_1\leq \frac{x}{d_1}\\d_2,m_1<q^{1+\alpha}D\\ (d_1d_2m_1,D q)=1}}1+O\left(\frac{x^{1-\frac \rho 2}}{Dq}\right),$$

$$\sum_{d_2<\mathcal B}\mathop{\sum\sum}\limits_{\substack{d_1m_1\leq \frac{x}{d_2}\\ d_1,m_1<q^{1+\alpha}D\\  d_1\geq\mathcal B \\ d_1d_2m_1\equiv a\pmod{Dq}}}\chi(d_1)=\frac{1}{\phi(Dq)}\sum_{d_2<\mathcal B}\mathop{\sum\sum}\limits_{\substack{d_1m_1\leq \frac{x}{d_2}\\ d_1,m_1<q^{1+\alpha}D\\  d_1\geq\mathcal B \\ (d_1d_2m_1,D q)=1}}\chi(d_1)+O\left(\frac{x^{1-\frac \rho 2}}{Dq}\right),$$
and
$$\sum_{m_1<\mathcal B}\mathop{\sum\sum}\limits_{\substack{d_1d_2\leq \frac{x}{m_1}\\ \mathcal B\leq d_1,d_2<q^{1+\alpha}D\\ d_1d_2m_1\equiv a\pmod{Dq}}}\chi(d_1)=\frac{1}{\phi(Dq)}\sum_{m_1<\mathcal B}\mathop{\sum\sum}\limits_{\substack{d_1d_2\leq \frac{x}{m_1}\\ \mathcal B\leq d_1,d_2<q^{1+\alpha}D\\  (d_1d_2m_1,Dq)=1}}\chi(d_1)+O\left(\frac{x^{1-\frac \rho 2}}{Dq}\right).$$

\end{lemma}

\begin{proof} For the first expression, we can apply Theorem \ref{binarydivisor}.  (We have stated it as the sum over all $n$, but it applies to subintervals such as the one here as well.)  So
\begin{align*}
\sum_{d_1<\mathcal B}\chi(d_1)\mathop{\sum\sum}\limits_{\substack{d_2m_1\leq \frac{x}{m_1}\\d_1d_2m_1\equiv a\pmod{Dq}}}1=&\frac{1}{\phi(Dq)}\sum_{\substack{d_1<\mathcal B }}\chi(d_1)\left(\mathop{\sum\sum}\limits_{\substack{d_2m_1\leq \frac{x}{d_1}\\(d_1d_2m_1,Dq)=1}}1+O\left(\frac{x}{d_1(Dq)^{1+\rho}}\right)\right)\\
=&\frac{1}{\phi(Dq)}\sum_{\substack{d_1<\mathcal B }}\chi(d_1)\mathop{\sum\sum}\limits_{\substack{d_2m_1\leq \frac{x}{d_1}\\(d_1d_2m_1,Dq)=1}}1+O\left(\frac{x^{1-\frac{\rho}{2}}}{Dq}\right).
\end{align*}

We next turn to the second equality in the lemma.  We split up the ranges $d_1\in [\mathcal B,Dq^{1+\alpha})$ and $m_1\in [\mathcal B,Dq^{1+\alpha})$ into dyadic intervals $[U,\zeta' U)$ and $[V,\zeta' Z)$, for some $1<\zeta'\leq 2$, and then we split these into subintervals $\mathcal U=[U_1,U_1+x^{-\rho}U)$ and $\mathcal V=[V_1,V_1+x^{-\rho}V)$.  Let $\overline{\mathcal U}$ and $\overline{\mathcal V}$ denote the set of these $U_1$ and $V_1$ that are the starting points of these smaller intervals.

Note that the only way that either $d_1$ or $m_1$ can be outside of these ranges is if $m_1<\mathcal B$, since neither one can be $\geq Dq^{1+\alpha}$, and $d_1\geq \mathcal B$.  In the case where $m_1<\mathcal B$, the sum will be negligible.  This is due to the fact that
$$\mathcal B^2Dq^{1+\alpha}=\frac{x^2q^{1+\alpha}}{D^5q^{3+2\alpha}}\leq \frac{x^{1-\frac \alpha 2}}{D^5},$$
since $q\geq \sqrt x$.  Hence by Shiu's theorem,
\begin{gather}\label{BBDq}\sum_{\substack{m_1<\mathcal B }}\sum_{\substack{d_2<\mathcal B }}\mathop{\sum}\limits_{\substack{d_1\leq Dq^{1+\alpha}\\d_1d_2m_1\equiv a \pmod{Dq}}}1\ll \sum_{\substack{n\leq \frac{x^{1-\frac \alpha 2}}{D^5} \\ n\equiv a\pmod{Dq}}}\tau_3(n)\ll \frac{x^{1-\frac \alpha 2}\log^2 x}{D^5\phi(Dq)}\ll \frac{x^{1-4\rho}}{Dq},\end{gather}
which is an acceptable error.  So we only need to consider these ranges  $d_1\in [\mathcal B,Dq^{1+\alpha})$ and $m_1\in [\mathcal B,Dq^{1+\alpha})$

Fix $d_2$, and let $y=\frac{x}{d_2}$.  Since $$(U_1+x^{-\rho}U)(V_1+x^{-\rho}V)\leq U_1V_1(1+2x^{-\rho}+x^{-2\rho}),$$we have

\begin{gather}\label{U1V1}\mathop{\sum\sum}\limits_{\substack{d_1m_1\leq y\\ d_1,m_1\in [\mathcal B,Dq^{1+\alpha})\\ d_1d_2m_1\equiv a\pmod{Dq}}}\chi(d_1)=\mathop{\sum\sum}_{\substack{U_1\in \overline{\mathcal U},V_1\in \overline{\mathcal V}\\U_1V_1\leq y}}\mathop{\sum\sum}\limits_{\substack{d_1\in \mathcal U,m_1\in \mathcal V \\ d_1d_2m_1\equiv a\pmod{Dq}}}\chi(d_1)+O\left(\sum_{\substack{y<n\leq y+3yx^{-\rho}\\ n\equiv a\overline{d_2}\pmod{Dq}}}\tau(n)\right).
\end{gather}
By Shiu's theorem, the big-O term is $$\ll \frac{yx^{-\rho}\log x}{Dq}.$$Moreover, by Lemma \ref{Weildyad} (or more specifically (\ref{dyadrewr})), if $(d_2,Dq)=1$ then
\begin{align*}
\mathop{\sum\sum}_{\substack{U_1\in \overline{\mathcal U},V_1\in \overline{\mathcal V}\\U_1V_1\leq y}}\mathop{\sum\sum}\limits_{\substack{d_1\in \mathcal U,m_1\in \mathcal V \\ d_1d_2m_1\equiv a\pmod{Dq}}}\chi(d_1)=&\frac{1}{\phi(Dq)}\mathop{\sum\sum}_{\substack{U_1\in \overline{\mathcal U},V_1\in \overline{\mathcal V}\\U_1V_1\leq y}}\left(\mathop{\sum\sum}\limits_{\substack{d_1\in \mathcal U,m_1\in \mathcal V \\ (d_1m_1,Dq)=1}}\chi(d_1)+O\left(\frac{yx^{-\rho}}{q^{1+\kappa-\alpha-2\eta}}+D^\frac 32q^{\frac 12+\frac \alpha 3}\right)\right)\\
=&\frac{1}{\phi(Dq)}\mathop{\sum\sum}_{\substack{U_1\in \overline{\mathcal U},V_1\in \overline{\mathcal V}\\U_1V_1\leq y}}\mathop{\sum\sum}\limits_{\substack{d_1\in \mathcal U,m_1\in \mathcal V \\ (d_1m_1,Dq)=1}}\chi(d_1)+O\left(\frac{y}{q^{1+\kappa-\frac{3\alpha}{2}-2\eta}}+D^\frac 32q^{\frac 12+\frac \alpha 2}\right)
\end{align*}
since $|\overline{\mathcal U}||\overline{\mathcal V}|\ll x^{2\rho}\log^2 x$.  Using the logic of (\ref{U1V1}) again in reverse, along with (\ref{BBDq}), we then have
\begin{gather*}\mathop{\sum\sum}\limits_{\substack{d_1m_1\leq y\\ d_1,m_1<q^{1+\alpha}D\\  d_1\geq\mathcal B \\ d_1d_2m_1\equiv a\pmod{Dq}}}\chi(d_1)=\frac{1}{\phi(Dq)}\mathop{\sum\sum}\limits_{\substack{d_1m_1\leq y\\ d_1,m_1<q^{1+\alpha}D\\  d_1\geq\mathcal B \\ (d_1d_2m_1,Dq)=1}}\chi(d_1)+O\left(\frac{yx^{-\rho}\log x}{\phi(Dq)}+\frac{y}{q^{1+\kappa-\frac{3\alpha}{2}-2\eta}}+D^\frac 32q^{\frac 12+\frac \alpha 2}\right).
\end{gather*}
So
\begin{align*}\sum_{d_2<\mathcal B}&\mathop{\sum\sum}\limits_{\substack{d_1m_1\leq \frac{x}{d_2}\\ d_1,m_1<q^{1+\alpha}D\\  d_1\geq\mathcal B \\ d_1d_2m_1\equiv a\pmod{Dq}}}\chi(d_1)\\
=&\sum_{d_2<\mathcal B}\left(\frac{1}{\phi(Dq)}\mathop{\sum\sum}\limits_{\substack{d_1m_1\leq \frac{x}{d_2}\\ d_1,m_1<q^{1+\alpha}D\\  d_1\geq\mathcal B \\ (d_1d_2m_1,Dq)=1}}\chi(d_1)+O\left(\frac{x^{1-\rho}\log x}{d_2\phi(Dq)}+\frac{x}{d_2q^{1+\kappa-\frac{3\alpha}{2}-2\eta}}+D^\frac 32q^{\frac 12+\frac \alpha 2}\right)\right)\\
=&\frac{1}{\phi(Dq)}\sum_{d_2<\mathcal B}\mathop{\sum\sum}\limits_{\substack{d_1m_1\leq \frac{x}{d_2}\\ d_1,m_1<q^{1+\alpha}D\\  d_1\geq\mathcal B \\ (d_1d_2m_1,Dq)=1}}\chi(d_1)+O\left(\frac{x^{1-\frac \rho 2}}{Dq}+\frac{x}{q^{1+\kappa-2\alpha-2\eta}}+\frac{x}{Dq^{1+\frac \alpha 2}}\right).
\end{align*}
Note that
$$\frac{x}{Dq^{1+\frac \alpha 2}}\leq \frac{x^{1-\frac \alpha 4}}{Dq}\leq \frac{x^{1-2\rho}}{Dq}.$$
since $q\geq \sqrt x$.  Likewise, our assumption that $\kappa>4\alpha+4\eta$ gives
$$\frac{x}{q^{1+\kappa-2\alpha-2\eta}}\leq  \frac{x}{Dq^{1+\kappa-2\alpha-4\eta}}\leq \frac{x}{Dq^{1+2\alpha}}\leq \frac{x^{1-8\rho}}{Dq}.$$
So we can consolidate the expression above:
\begin{align*}\sum_{d_2<\mathcal B}&\mathop{\sum\sum}\limits_{\substack{d_1m_1\leq \frac{x}{d_2}\\ d_1,m_1<q^{1+\alpha}D\\  d_1\geq\mathcal B \\ d_1d_2m_1\equiv a\pmod{Dq}}}\chi(d_1)=\frac{1}{\phi(Dq)}\sum_{d_2<\mathcal B}\mathop{\sum\sum}\limits_{\substack{d_1m_1\leq \frac{x}{d_2}\\ d_1,m_1<q^{1+\alpha}D\\  d_1\geq\mathcal B \\ (d_1d_2m_1,Dq)=1}}\chi(d_1)+O\left(\frac{x^{1-\frac \rho 2}}{Dq}\right).\end{align*}
This completes the proof for the second expression in the lemma.  The proof for the third expression is the same as the second.
\end{proof}

\section{A Binary Sum Corollary}

An interesting corollary to the above is that we can find a sharper version of the Brun-Titchmarsh theorem for $q\leq x^{\frac 23-\varepsilon}$ for some $\varepsilon>0$ if we restrict our primes to those $p$ for which $\chi(p)=1$.  We prove this in two steps.  First, we find a bound for a sum of $\lambda(n)$ in an arithmetic progression.
\begin{corollary}\label{dyadcor2}
Let $q$ and $x$ be such that $q\leq \frac{x^{\frac 23-\alpha}}{D^\frac 32}$, and let $(a,q)=1$.  Then
$$\sum_{\substack{n\leq x\\n\equiv a\pmod{q}}}\lambda(n)\ll \frac{x}{q}L(1,\chi)\log x.$$

\end{corollary}
\begin{proof}
We will assume that $\sqrt x<q\leq \frac{x^{\frac 23-\alpha}}{D^\frac 32}$; if it is not, one can multiply $q$ by a prime $s$ with $(a,s)=1$ so that $qs$ is in this range.  We begin by working with the sum mod $Dq$, which is to say the sum
$$\sum_{\substack{n\leq x\\n\equiv a\pmod{Dq}}}\lambda(n).$$
We will assume that $(a,D)=1$ as well.  If it were not, there would be another $b$ such that $b\equiv a\pmod q$ and $(b,D)=1$.  Hence we can assume without loss of generality that this $b$ is actually the $a$ we have chosen.

Obviously, this is the same as working with the sum
$$\mathop{\sum\sum}_{\substack{jl\leq x\\ jl\equiv a\pmod{Dq}}}\chi(j).$$
Let $\kappa>0$ be as in the previous two lemmas.  For the restriction
$$\mathop{\sum\sum}_{\substack{jl\leq x \\q^\kappa\leq j,l\leq xq^{-\kappa} \\ jl\equiv a\pmod{Dq}}}\chi(j),$$
we can use exactly the same process that we used in (\ref{dyadrewr2}), finding
$$\mathop{\sum\sum}_{\substack{jl\leq x \\q^\kappa\leq j,l\leq xq^{-\kappa} \\ jl\equiv a\pmod{Dq}}}\chi(j)\ll \frac{x}{q^{1+\kappa-2\alpha-\eta}}+D^\frac 32q^{\frac 12+\frac \alpha 2}\ll \frac{x^{1-\frac \rho 2}}{Dq}+D^\frac 32q^{\frac 12+\frac \alpha 2}.$$
Note that the right-hand side is smaller than the bound in the corollary here.  Thus, we will be able to ignore it.

So it remains to evaluate only the edges of the interval.  In the case where $j$ is small, we can apply Lemma \ref{chiId2}:
$$\sum_{j<q^\kappa}\chi(j)\sum_{\substack{l\leq \frac xj\\ l\equiv a\bar j\pmod{Dq}}}1=\frac x{Dq}\sum_{\substack{j<q^\kappa \\ (j,q)=1}}\frac{\chi(j)}{j}+O\left(q^\kappa\right)\ll \frac{x}{Dq}L(1,\chi)\log x .$$
In the case where $j$ is large and $l$ is small, we note that since $j\equiv a\bar l\pmod D$, we have $\chi(j)=\chi(a)\chi(l)$.  So
$$\sum_{l<q^\kappa}\sum_{\substack{j\leq \frac xl\\ j\equiv a\bar l\pmod{Dq}}}\chi(j)=\chi(a)\sum_{l<q^\kappa}\chi(l)\sum_{\substack{j\leq \frac xl\\ j\equiv a\bar l\pmod{Dq}}}1.$$
From here, the proof is the same as in the case of $j$ small.

To change the sum mod $Dq$ into a sum mod $q$, we sum the congruence classes mod $Dq$ that will give us $a$ mod $q$:
$$\mathop{\sum}\limits_{\substack{n\leq x \\ n\equiv a\pmod{q}}}\lambda(n)=\sum_{\substack{0\leq a'<D \\ (a+a'q,D)=1}}\mathop{\sum}\limits_{\substack{n\leq x \\ \phantom{1}n\equiv a+a'q\pmod{Dq}\phantom{1}}}\lambda(n)\ll\frac{x}{q}L(1,\chi)\log x.$$

\end{proof}

From here, we can give a version of the Brun-Titchmarsh inequality among the exceptional primes in the presence of Siegel zeroes.

\begin{corollary}\label{BTforprimes}
If $q\leq \frac{x^{\frac 23-\alpha}}{D^\frac 32}$ then for any $a$ with $(a,q)=1$,
$$\sum_{\substack{ p\leq x \\  p\equiv a \pmod{q}\\\chi(p)=1 }}1\ll \frac{x}{q}L(1,\chi)\log x.$$
\end{corollary}
\begin{proof}
Since $\lambda(p)=2$ for these primes, we see that
$$\sum_{\substack{ p\leq x \\  p\equiv a \pmod{q} \\ \chi(p)=1 }}1\leq \frac 12 \sum_{\substack{ n\leq x \\  n\equiv a \pmod{q} }}\lambda(n).$$
The corollary then follows from the previous corollary.
\end{proof}

\section{Ternary Sums}\label{tersumsec1}

For the remaining intervals, we will need to deal with the full twisted ternary sum.  To do this, we will apply the theorems about Kloosterman sums from Friedlander, Iwaniec, and Shparlinski.

Recall that we are examining the sum
\begin{gather}\label{tersum}\mathop{\sum\sum\sum}_{\substack{d_1d_2m_1\leq x \\ \mathcal B\leq d_1,d_2,m_1<Dq^{1+\alpha} \\ d_1d_2m_1\equiv a\pmod{Dq}}}\chi(d_1).\end{gather}
As before, we split the region $[\mathcal B,Dq^{1+\alpha})$ into dyadic intervals $[U,\zeta' U)$, and then those are split into subregions $[M,M+\Delta U)$ for some $\Delta=o(1)$.  In particular, we will take $\Delta=x^{-\nu}$ where $\nu<1/10000$ and $D\leq x^{\eta}\leq x^{\nu/5}$.  Let $$\mathcal M=\{[M_1,M_1+\Delta U_1)\times [M_2,M_2+\Delta U_2)\times [M_3,M_3+\Delta U_3))\}$$ denote a set of ordered triples where the $U_i$ are starting points of the dyadic regions and the $M_i$ are starting points for the subregions, and let $\overline{\mathcal M}$ be the set of all such $\mathcal M$ in (\ref{tersum}).

We consider sums of the form
$$\mathcal T(\mathcal M)=\mathop{\sum\sum\sum}_{\substack{(t_1,t_2,t_3)\in \mathcal M\\ t_1t_2t_3\equiv a\pmod{Dq}}}f_1(t_1)f_2(t_2)f_3(t_3)$$
and
$$\mathcal T^*(\mathcal M)=\mathop{\sum\sum\sum}_{\substack{(t_1,t_2,t_3)\in \mathcal M \\ (t_1t_2t_3,Dq)=1}}f_1(t_1)f_2(t_2)f_3(t_3),$$
where one of the $f_i$ is $\chi$ and the other two are 1.  Hereafter, we will assume without loss of generality that $M_1\geq M_2\geq M_3$.


We can then write
$$\mathop{\sum\sum\sum}_{\substack{d_1d_2m_1\leq x \\ \mathcal B\leq d_1,d_2,m_1<Dq^{1+\alpha} \\ d_1d_2m_1\equiv a\pmod{Dq}}}\chi(d_1)=\sum_{\substack{\mathcal M\in  \overline{\mathcal M}\\ \frac{L(1,\chi)x}{\log^{10}x}\leq M_1M_2M_3\leq x}}\mathcal T(\mathcal M)+O\left(\sum_{\substack{x<n\leq (1+\Delta)^3x\\ n\equiv a\pmod{Dq}}}\tau_3(n)+\sum_{\substack{n\leq \frac{8xL(1,\chi)}{\log^{10} x}\\ n\equiv a\pmod{Dq}}}\tau_3(n)\right).$$
By Shiu's theorem, the first big-O term can be bounded with
$$\sum_{\substack{x<n\leq (1+\Delta)^3x\\ n\equiv a\pmod{Dq}}}\tau_3(n)\ll \frac{x\Delta\log^2 x}{Dq}\ll \frac{x^{1-\frac{\nu}{2}}}{Dq},$$
while the second can be bounded with
$$\sum_{\substack{n\leq \frac{8L(1,\chi)x}{\log^{10} x}\\ n\equiv a\pmod{Dq}}}\tau_3(n)\ll \frac{L(1,\chi)x}{Dq\log^8 x}.$$
The latter of the two estimates is larger if $x$ is sufficiently large.  So we have
\begin{gather}\label{log10bd}
\mathop{\sum\sum\sum}_{\substack{d_1d_2m_1\leq x \\ \mathcal B\leq d_1,d_2,m_1<Dq^{1+\alpha} \\ d_1d_2m_1\equiv a\pmod{Dq}}}\chi(d_1)=\sum_{\substack{\mathcal M\in \overline{\mathcal M} \\ \frac{L(1,\chi)x}{\log^{10}x}\leq M_1M_2M_3\leq x}}\mathcal T(\mathcal M)+O\left(\frac{L(1,\chi)x}{Dq\log^8 x}\right).
\end{gather}
Similarly,
\begin{gather}\label{log10bd2}
\mathop{\sum\sum\sum}_{\substack{d_1d_2m_1\leq x \\ \mathcal B\leq d_1,d_2,m_1<Dq^{1+\alpha} \\ (d_1d_2m_1,Dq)=1}}\chi(d_1)=\frac{1}{\phi(Dq)}\sum_{\substack{\mathcal M\in \overline{\mathcal M} \\ \frac{L(1,\chi)x}{\log^{10}x}\leq M_1M_2M_3\leq x}}\mathcal T^*(\mathcal M)+O\left(\frac{L(1,\chi)x}{Dq\log^8 x}\right).
\end{gather}

Our goal will then be to show that
$$\sum_{\substack{\mathcal M\in \overline{\mathcal M} \\ \frac{L(1,\chi)x}{\log^{10}x}\leq M_1M_2M_3\leq x}}\mathcal T(\mathcal M)=\frac{1}{\phi(Dq)}\sum_{\substack{\mathcal M\in \overline{\mathcal M} \\ \frac{L(1,\chi)x}{\log^{10}x}\leq M_1M_2M_3\leq x}}\mathcal T^*(\mathcal M)+O\left(\frac{x^{1-\varepsilon}}{Dq}\right)$$
for some $\varepsilon>0$.  Since $\left|\overline{\mathcal M}\right|\ll x^{3\nu}\log^3 x$, it will suffice to show that
$$\left|\mathcal T(\mathcal M)-\frac{1}{\phi(Dq)}\mathcal T^*(\mathcal M)\right|\ll \frac{x^{1-4\nu}}{Dq},$$
which will then give us
\begin{gather}\label{Tbound1}
\sum_{\mathcal M\in \overline{\mathcal M}}\left|\mathcal T(\mathcal M)-\frac{1}{\phi(Dq)}\mathcal T^*(\mathcal M)\right|\ll \frac{x^{1-\nu}\log^3 x}{Dq}.
\end{gather}

We note, however, that if we are to evaluate $\mathcal T^*(\mathcal M)$, we find significant cancellation for whichever $i$ has $f_i=\chi$.  Indeed, the sum over this $t_i$ will be bounded by $Dq^{\varepsilon}$ for any $\varepsilon>0$ by Lemma \ref{chiId2}.  So along the remaining interval with $M_3\geq \frac{x}{D^\frac 52q^{\frac 32+\alpha}}$, we have
$$\left|\frac{1}{\phi(Dq)}\mathcal T^*(\mathcal M)\right|\ll \frac{Dq^\varepsilon}{\phi(Dq)}\Delta^2 U_1U_2\ll \frac{Dq^\varepsilon \Delta^2 x}{\phi(Dq)M_3}\ll \frac{D^\frac 72q^{\frac 32+2\alpha}\Delta^2}{\phi(Dq)}.$$
Since Theorem \ref{T1final} will take $q<x^{\frac{16}{31}}$, the above bound is clearly much smaller than $\frac{x^{1-4\nu}}{Dq}$.  So it will suffice to simply show that $\left|\mathcal T(\mathcal M)\right|\ll \frac{x^{1-4\nu}}{Dq}$.

Thus, our goal here will be to prove the following.

\begin{theorem}\label{mainishgoal}
For $\mathcal M\in \overline{\mathcal M}$ as defined above and $q<x^{\frac{30}{59}-12\nu}$,
\begin{align*} 
\left|\mathcal T(\mathcal M)\right|\ll \frac{x^{1-4\nu}}{Dq}.
\end{align*}
Consequently,
$$\mathop{\sum\sum\sum}_{\substack{d_1d_2m_1\leq x \\ \mathcal B\leq d_1,d_2,m_1<Dq^{1+\alpha} \\ d_1d_2m_1\equiv a\pmod{Dq}}}\chi(d_1)=\frac{1}{\phi(Dq)}\mathop{\sum\sum\sum}_{\substack{d_1d_2m_1\leq x \\ \mathcal B\leq d_1,d_2,m_1<Dq^{1+\alpha} \\ (d_1d_2m_1,Dq)=1}}\chi(d_1)+O\left(\frac{L(1,\chi)x}{Dq\log^8 x}\right).$$
\end{theorem}
The latter part follows from the fact that the bound in (\ref{log10bd}) is larger than the bound in (\ref{Tbound1}).

We will also prove that if we are willing to sum over a range of $q\sim Q$, we can increase the allowable range for $Q$.  This will be the second of our goals for next few sections.  We will use the notation $\mathcal T_{Dq}(\mathcal M)$ and $\mathcal T_{Dq}^*(\mathcal M)$ to indicate the fact that the sums $T$ and $T^*$ depend on a congruence modulo $Dq$ or a relative primality with $Dq$.

Here, however, we note that we must depart slightly from our previous work and take $\mathcal B=\frac{x}{D^\frac 52 Q^{\frac 32+\alpha}}$ (rather than $\frac{x}{D^\frac 52 q^{\frac 32+\alpha}}$) and partition the interval $[\mathcal B,DQ^{1+\alpha})$ for our subintervals $M_i$ to ensure that all $T_{Dq}$ and $T_{Dq}^*$ have the same intervals $M_i$.  This change from $q$ to $Q$ does not cause any significant change in the bounds in Lemmas \ref{easy1}, \ref{easy2}, or \ref{binarylemma}, since $q$ and $Q$ differ by at most a factor of 2.

We will prove the following.
\begin{theorem}\label{mainishgoal2}
Let $Q<x^{\frac{16}{31}-12\nu}$.   Then
\begin{align*}
\sum_{q\sim Q}\left|\mathcal T_{Dq}(\mathcal M)\right|\ll \frac{x^{1-4\nu}}{D}.
\end{align*}
Consequently,
$$\sum_{q\sim Q}\left|\mathop{\sum\sum\sum}_{\substack{d_1d_2m_1\leq x \\ \mathcal B\leq d_1,d_2,m_1<DQ^{1+\alpha} \\ d_1d_2m_1\equiv a\pmod{Dq}}}\chi(d_1)-\frac{1}{\phi(Dq)}\mathop{\sum\sum\sum}_{\substack{d_1d_2m_1\leq x \\ \mathcal B\leq d_1,d_2,m_1<DQ^{1+\alpha} \\ (d_1d_2m_1,Dq)=1}}\chi(d_1)\right|\ll \frac{L(1,\chi)x}{D\log^8 x}.$$
\end{theorem}

\section{Ideas of Friedlander and Iwaniec: The Error Term}

For the remaining regions, we use the framework of \cite{FI85}.  First, define
$$c(h)=\frac 1{Dq}\int_{M_1}^{M_1+\Delta U_1}e\left(\frac{hz}{Dq}\right)dz.$$
By (3.2) of \cite{FI85}, for any $H$ with $0<H<Dq$,
\begin{gather}\label{congclass}\sum_{\substack{A\leq t_1\leq B \\ t_1\equiv a\bar t_2\bar t_3\pmod{Dq}}}1=\frac{B-A}{Dq}+\sum_{0<|h|\leq H}c(h)e\left(-\frac{a\bar t_2\bar t_3 h}{Dq}\right)+O\left(\varrho\left(\frac{B-a\bar t_2\bar t_3}{Dq}\right)-\varrho\left(\frac{A-a\bar t_2\bar t_3}{Dq}\right)\right),\end{gather}
where for any real number $z$,
$$\varrho(z)=\min(1,(H||z||)^{-1})$$
and $||\cdot ||$ again indicates distance to the nearest integer.  Lemma 3.1 of \cite{FI85} states that for any $V\geq 1$ and any $\varepsilon>0$,
$$\sum_{\substack{K'\leq k\leq K'+K \\ (k,q)=1}}\varrho\left(\frac{V-a\bar k}{Dq}\right)\ll \left(\frac{K'}{H}+(Dq)^\frac 12+\frac K{Dq}\right)x^{\varepsilon}$$
From this, it follows that,
\begin{gather}\label{rhoerror}\sum_{\substack{M_2\leq t_2\leq (1+\Delta)M_2 \\ M_3\leq t_3\leq (1+\Delta)M_3 \\ (t_2t_3,q)=1}}\varrho\left(\frac{V-a\bar t_2\bar t_3}{Dq}\right)\leq \sum_{\substack{M_2M_3\leq k\leq (1+\Delta)^2M_2M_3 \\ (k,q)=1}}\varrho\left(\frac{V-a\bar k}{Dq}\right)\tau(k)\ll \left(\frac{M_2M_3}{H}+(Dq)^\frac 12+\frac{M_2M_3}{Dq}\right)x^{\nu}.\end{gather}
Note that the first term on the right-hand side is larger than the third, since $H<qD$. Here, we will take
$$H=\frac{D^2qM_2M_3}{x^{1-5\nu}}\leq \frac{D^2qx^{5\nu}}{M_1}$$
for any choice of $\mathcal M$.  We can thus rewrite (\ref{rhoerror}) as
$$\sum_{\substack{M_2\leq t_2\leq (1+\Delta)M_2 \\ M_3\leq t_3\leq (1+\Delta)M_3 \\ (t_2t_3,q)=1}}\varrho\left(\frac{V-a\bar t_2\bar t_3}{Dq}\right)\ll \frac{x^{1-4\nu}}{D^2q}+(Dq)^\frac 12x^\nu.$$


\section{Bounding $G$}\label{tersumsec3}

We bound separately the cases where $f_1=\chi$ and $f_1=1$. We begin with the former.
\begin{lemma}\label{Gchi}
Let $f_1=\chi$, and define
$$\mathcal G(H,\mathcal M,\chi,y,j)=\max_{(a,Dq)=1}\left|\sum_{1\leq |h|\leq y}\sum_{\substack{M_2<t_2\leq M_2+\Delta U_2\\ M_3<t_3\leq M_3+\Delta U_3 \\ (t_2t_3,Dq)=1}}e\left(-\frac{a(h+jq)\bar t_2\bar t_3}{Dq}\right)\right|,$$
and 
$$\mathcal G(H,\mathcal M,\chi)=\max_{1\leq y\leq H}\max_{0\leq j\leq D-1}\mathcal G(H,\mathcal M,\chi,y,j).$$
Then
\begin{align*}
\left|\mathcal T(\mathcal M)\right|\ll \frac{M_1x^{4\nu}}{D^4q}\mathcal G(H,\mathcal M,\chi)+\frac{x^{1-4\nu}}{Dq}+D^\frac 32q^\frac 12 x^\nu.
\end{align*}
\end{lemma}
\begin{proof}Let $q'$ denote the largest divisor of $q$ that is relatively prime to $D$.  Then
\begin{align*} 
\mathcal T(&\mathcal M)=\mathop{\sum\sum\sum}_{\substack{(t_1,t_2,t_3)\in \mathcal M\\ t_1t_2t_3\equiv a\pmod{Dq}}}\chi(t_1)\\
=&\sum_{b=1}^D\chi(b)\mathop{\sum\sum}_{\substack{t_2,t_3\\ M_i\leq t_i<M_i+\Delta U_i\\ t_2t_3\equiv a\bar b\pmod{D}\\(t_2t_3,q')=1}}\sum_{\substack{M_1\leq t_1<M_1+\Delta U_1 \\ t_1\equiv a\overline{t_2t_3}\pmod{Dq}}}1\\
=&\sum_{b=1}^D\chi(b)\left(\mathop{\sum\sum}_{\substack{t_2,t_3\\ M_i\leq t_i<M_i+\Delta U_i\\ t_2t_3\equiv a\bar b\pmod{D}\\(t_2t_3,q')=1}}\left(\frac{\Delta U_1}{Dq}+\sum_{0<|h|\leq H}c(h)e\left(-\frac{a\bar t_2\bar t_3 h}{Dq}\right)\right)+O\left(\frac{x^{1-4\nu}}{D^2q}+(Dq)^\frac 12x^\nu\right)\right)\\
\end{align*}
Note that $\chi(b)=\chi(t_2)\chi(t_3)\chi(a)$ by the congruence mod $D$, and that $U_1\leq Dq^{1+\alpha}$.  So we apply Lemma \ref{chiId2}:
$$\left|\sum_{b=1}^D\chi(b)\mathop{\sum\sum}_{\substack{t_2,t_3\\ M_i\leq t_i<M_i+\Delta U_i\\ t_2t_3\equiv a\bar b\pmod{D}\\(t_2t_3,q')=1}}\frac{\Delta U_1}{Dq}\right|=\left|\chi(a)\mathop{\sum\sum}_{\substack{t_2,t_3\\ M_i\leq t_i<M_i+\Delta U_i\\(t_2t_3,q')=1}}\chi(t_2)\chi(t_3)\sum_{\substack{1\leq b\leq D-1\\b\equiv a\overline{t_2t_3}\pmod D}}\frac{\Delta U_1}{Dq}\right|\ll \frac{D\Delta U_1}{q^{1-\alpha}}\leq D^2q^{2\alpha}.$$
Hence,
\begin{align*}
\left|\mathcal T(\mathcal M)\right|\ll &\left|\frac 1D\sum_{b=1}^D\chi(b)\sum_{j=0}^{D-1}e\left(\frac{jb}{D}\right)\mathop{\sum\sum}_{\substack{t_2,t_3\\ M_i\leq t_i<M_i+\Delta U_i\\(t_2t_3,Dq)=1}}\sum_{0<|h|\leq H}c(h)e\left(-\frac{a\bar t_2\bar t_3 h}{Dq}\right)e\left(-\frac{aj\bar t_2\bar t_3}{D}\right)\right|\\
&+\left(\frac{x^{1-4\nu}}{Dq}+D^\frac 32q^\frac 12 x^\nu\right),
\end{align*}
where we use an exponential sum to detect the congruence mod $D$.
Define
$$G(H,\mathcal M,\chi)=\frac 1D\sum_{b=1}^D\left|\sum_{j=1}^De\left(\frac{jb}{D}\right)\mathop{\sum\sum}_{\substack{t_2,t_3\\ M_i\leq t_i<M_i+\Delta U_i \\(t_2t_3,Dq)=1}}\sum_{1\leq |h|\leq H}c(h)e\left(-\frac{a\bar t_2\bar t_3 (h+jq)}{Dq}\right)\right|.$$
It then remains to relate $G$ to $\mathcal G$ here.

By partial summation over $h$,
\begin{align*}
G(H,\mathcal M,\chi)\ll  & D|c(H)|\max_{0\leq j\leq D-1}\left|\sum_{1\leq |h|\leq H}\sum_{\substack{M_2<t_2\leq M_2+\Delta U_2\\ M_3<t_3\leq M_3+\Delta U_3 \\ (t_2t_3,Dq)=1}}e\left(-\frac{a\bar t_2\bar t_3 (h+jq)}{Dq}\right)\right|\\
&+D\max_{0\leq j\leq D-1}\int_1^H \max_{0\leq w\leq y}|c'(w)|\left|\sum_{1\leq |h|\leq y}\sum_{\substack{M_2<t_2\leq M_2+\Delta U_2\\ M_3<t_3\leq M_3+\Delta U_3 \\ (t_2t_3,Dq)=1}}e\left(-\frac{a\bar t_2\bar t_3 (h+jq)}{Dq}\right)\right|dy.
\end{align*}
Using the trivial bound of $\left|e\left(\frac{hz}{Dq}\right)\right|\leq 1$, we have $$D|c(H)|\ll \frac{M_1x^{-\nu}}{q}.$$
Meanwhile, $|c'(w)|\ll (Dq)^{-2}\Delta^2 M_1^2$, which means that 
$$D\int_1^H \max_{0\leq w\leq y}|c'(w)|dy\ll DH(Dq)^{-2}\Delta^2 M_1^2\leq \frac{DM_1x^{3\nu}}{q}.$$
The second bound is obviously larger, and hence
$$G(H,\mathcal M,\chi)\ll \frac{DM_1x^{3\nu}}{q}\left(\max_{1\leq y\leq H}\max_{0\leq j\leq D-1}\mathcal G(H,\mathcal M, \chi,y,j)\right).$$
Plugging this back into our bound for $\mathcal T$, we have
\begin{align*}
\left|\mathcal T(\mathcal M)\right|\ll \frac{D M_1x^{3\nu}}{q}\mathcal G(H,\mathcal M, \chi)+\frac{x^{1-4\nu}}{Dq}+D^\frac 32q^\frac 12x^\nu.
\end{align*}
Note that $D\leq \frac{x^\nu}{D^4}$ by assumption.  So we can simplify the first term on the right to be
$$\frac{M_1x^{4\nu}}{D^4q}\mathcal G(H,\mathcal M, \chi).$$
\end{proof}
\begin{lemma}\label{G1}
Let $f_1=1$, and define as before
$$\mathcal G(H,\mathcal M,1)=\max_{y\leq H}\max_{(a,Dq)=1}\left|\sum_{1\leq |h|\leq y}\sum_{\substack{M_2<t_2\leq M_2+\Delta U_2\\ M_3<t_3\leq M_3+\Delta U_3 \\ (t_2t_3,Dq)=1}}e\left(-\frac{a\bar t_2\bar t_3h}{Dq}\right)f_2(t_2)f_3(t_3)\right|.$$
Then
\begin{align*}
\left|\mathcal T(\mathcal M)\right|\ll \frac{M_1x^{4\nu}}{D^5q}\mathcal G(H,\mathcal M,1)+\frac{x^{1-4\nu}}{Dq}+D^\frac 32q^{\frac 12+2\alpha}x^{-2\nu}.
\end{align*}
\end{lemma}
\begin{proof}
The steps are similar.  Here, we have
\begin{align*}
\mathcal T(&\mathcal M)=&\mathop{\sum\sum}_{\substack{t_2,t_3\\ M_i\leq t_i<M_i+\Delta U_i\\ (t_2t_3,Dq)=1}}f_2(t_2)f_3(t_3)\left(\frac{\Delta U_1}{Dq}+\sum_{0<|h|\leq H}c(h)e\left(-\frac{a\bar t_2\bar t_3 h}{Dq}\right)\right)+O\left(\frac{x^{1-4\nu}}{D^2q}+(Dq)^\frac 12\right)\\
\end{align*}
Since one of $f_2$, $f_3$ is $\chi$ and the other is 1, we have
$$\left|\mathop{\sum\sum}_{\substack{t_2,t_3\\ M_i\leq t_i<M_i+\Delta U_i\\ (t_2t_3,Dq)=1}}f_2(t_2)f_3(t_3)\frac{\Delta U_1}{Dq}\right|\ll \frac{M_1M_2\Delta^2}{Dq^{1-\alpha}}\ll D^{\frac 32}q^{\frac 12+2\alpha}\Delta^2$$
since $M_1M_2\leq \frac{x}{M_3}$ and $M_3\geq \frac{x}{D^\frac 52q^{\frac 32+\alpha}}$.

From here, the steps are identical to the previous theorem except that we save a factor of $D$ and no longer have to take the max over $j$ because we have not removed the character and hence do not sum over $b$ or $j$.  Hence the lemma follows.
\end{proof}




For ease of notation, let $R=Dq$. Since
$$\frac{M_1x^{4\nu}}{D^4q}\mathcal G(H,\mathcal M, \chi)=\frac 1R\left(\frac{M_1x^{4\nu}}{D^3}\mathcal G(H,\mathcal M, \chi)\right),$$
these two lemmas then lead us to the following goal:


\begin{theorem}\label{Ggoal} For $f_1=1$ or $\chi$,
$$\frac{M_1}{D^3}\mathcal G(H,\mathcal M,f_1)\ll x^{1-8\nu}.$$
\end{theorem}

The authors of \cite{FI85} use two bounds for the exponential sum over $h$, $t_2$, and $t_3$.  In this paper, we use one of the two bounds, but we replace the other with the results of Shparlinski \cite{Shp}.

The bound of \cite{FI85} that we keep is (2.6) in that paper, which we stated as Theorem \ref{FI:2.6} above.  We plug in here $M=M_3$, $N=M_2$.
\begin{LemFI}
Let $(a,R)=1$.  Then for any $\varepsilon>0$,
\begin{align*}K_\chi(a)\ll & Dx^\varepsilon \left(R^\frac 34 H^{\frac 12}M_3^{\frac 12} + R^\frac 14HM_3 + R^\frac 13H^\frac 23M_2^\frac 23M_3^\frac 13+HM_2^\frac 23M_3^\frac 23+R^{-1}HM_2M_3\right)
\end{align*}
\end{LemFI}
In the case of a fixed $q$, we replace the other \cite{FI85} bound ((2.5) in that paper) with the Shparlinski result, which appeared above as Theorem \ref{Shp:T1}.  In this case, we plug in $M=M_2$, $N=M_3$

\begin{ThSh1} Let $(a,R)=1$.  Then
$$K_\chi(a)\ll \left(HM_2+(HM_2)^\frac 34R^\frac 14\right)\left(M_3^\frac 78R^{-\frac 18}+M_3^\frac 12\right)R^{o(1)}.$$
\end{ThSh1}

\section{The Friedlander-Iwaniec Method for Somewhat Large $M_2$}\label{M2sec}
Recall that in our new terminology, we have
$$H=\frac{DRM_2M_3}{x^{1-5\nu}}\leq \frac{DRx^{5\nu}}{M_1}$$
By the fact that $M_1M_2M_3$ is bounded below, we then have
$$H\geq \frac{DRx^{5\nu}L(1,\chi)}{M_1\log^{10} x}\gg \frac{DRx^{4\nu}}{M_1}.$$
We handle first the case where $M_2>\frac{R^{\frac 52}}{x^{1-24\nu}}$.  In this case, we apply the bound for $K_\chi(a)$ as given in Theorem \ref{FI:2.6}.
\begin{lemma}  Let $R<x^{\frac{17}{32}-\alpha}$.  
If $M_2>\frac{R^{\frac 52}}{x^{1-24\nu}}$ then
$$\frac{M_1}{D^3}\mathcal G(H,\mathcal M,f_1)\ll x^{1-8\nu}.$$
\end{lemma}
\begin{proof}
Applying Theorem \ref{FI:2.6} with $M=M_2$, $N=M_3$, taking $\varepsilon=\nu$, and recalling that $\eta<\frac \nu 5$,
\begin{align*}\frac{M_1}{D^3}\mathcal G(H,\mathcal M,f_1)&\ll  \frac{M_1x^\nu}{D^3}\left(R^\frac 34 H^{\frac 12}M_3^{\frac 12} + R^\frac 14HM_3 + R^\frac 13H^\frac 23M_2^\frac 23M_3^\frac 13+HM_2^\frac 23M_3^\frac 23+R^{-1}HM_2M_3\right)\\
\ll  &\frac{Rx^{6\nu}}{D^2}\left(R^\frac 34 H^{-\frac 12}M_3^{\frac 12} + R^\frac 14M_3 + R^\frac 13H^{-\frac 13}M_2^\frac 23M_3^\frac 13+M_2^\frac 23M_3^\frac 23+R^{-1}M_2M_3\right)\\
\ll  &\frac{x^{6\nu}}{D^2} \left(\frac{R^\frac 54 M_1^{\frac 12}M_3^\frac 12}{D^\frac 12 x^{2\nu}} + R^\frac 54M_3 + Rx^{\frac 13-\nu}M_2^\frac 13+RM_2^\frac 23M_3^\frac 23+M_2M_3\right)
\end{align*}
Note that since $M_3\geq \frac{x}{R^{\frac 32+\alpha}D}$ by Lemma \ref{binarylemma}, we can use our assumed bound for $R$ to find
$$M_2\leq \sqrt{M_1M_2}\leq \sqrt{\frac{x}{M_3}}\leq R^{\frac 34+\frac \alpha 2}D^\frac 12<x^{\frac{51}{128}}D^\frac 12.$$
Moreover, by our assumption on $M_2$, we have
$$M_1M_3\leq \frac{x}{M_2}< \frac{x^{2-24\nu}}{R^\frac 52}.$$
We can plug these two bounds into the above.  We also apply the loose bounds of $M_1\leq x^{\frac 13}$, $M_1M_2\leq x^\frac 23$, and $R<x^{\frac{17}{32}}$ to the other terms:
\begin{align*}
\frac{M_1}{D^3}\mathcal G(H,\mathcal M,f_1)\ll  &\frac{x^{6\nu}}{D^2}\left(\frac{x^{1-12\nu}}{D^\frac 12 x^{2\nu}} + x^\frac{85}{128}x^\frac 13 + x^{\frac{17}{32}}x^{\frac 13-\nu}x^\frac{17}{128}D^{\frac 16}+x^\frac{17}{32}x^\frac 49+x^\frac 23\right)\\
\ll  &x^{1-8\nu}+ x^{\frac{383}{384}+6\nu} + x^{\frac{383}{384}+5\nu} +x^{\frac{281}{288}+6\nu}+x^{\frac 23+6\nu}.
\end{align*}
Note that $\frac{383}{384}+14\nu<\frac{383}{384}+.0014\leq .9988$, and hence all of these exponents are $\leq 1-8\nu$.  The lemma then follows.

\end{proof}

\section{The Shparlinski Method for Fixed $q$}
For the remaining cases of a fixed $q$ with $M_2\leq \frac{R^{\frac 52}}{x^{1-24\nu}}$, we apply the bounds of Shparlinski.  This will result in the following theorem:
\begin{theorem}\label{notonavg}
If $q\leq x^{\frac{30}{59}-12\nu}$ then 
$$\frac{M_1}{D^3}\mathcal G(H,\mathcal M,f_1)\ll x^{1-10\nu}.$$
\end{theorem}

\begin{proof}
Applying Theorem \ref{Shp:T1} to $\mathcal G$, we let $M=M_2$, $N=M_3$, and we replace $R^{o(1)}$ with $x^\eta$:
\begin{align*}\frac{M_1}{D^3}\mathcal G(H,\mathcal M,f_1)&\ll  \frac{M_1}{D^3}\left(HM_2+(HM_2)^\frac 34R^\frac 14\right)\left(M_3^\frac 78R^{-\frac 18}+M_3^\frac 12\right)x^\eta
\end{align*}
Note that $M_3^\frac 78R^{-\frac 18}\geq M_3^\frac 12$ as long as $M_3>R^\frac 13$.  This is always true here, since $R<x^\frac{12}{23}$ implies$$R^\frac 13<x^\frac 4{23}\leq \frac{x}{q^{\frac{19}{12}}}< \mathcal B.$$
Moreover, since
$$HM_2=x^{5\nu}\frac{DRM_2^2M_3}{x}\leq x^{5\nu}\frac{DRM_1M_2M_3}{x}\leq Dx^{5\nu} R,$$
we can see that 
\begin{gather}\label{HM2}
HM_2\leq x^{2\nu}(HM_2)^\frac 34R^\frac 14.
\end{gather}
Hence, we can simplify the inequality for $\mathcal G$ to
\begin{align*}\frac{M_1}{D^3}\mathcal G(H,\mathcal M,f_1)&\ll  \frac{M_1}{D^3}\left(x^{2\nu}(HM_2)^\frac 34R^\frac 14\right)\left(M_3^\frac 78R^{-\frac 18}\right)x^\eta
\end{align*}
Plugging in for $H$, we then have
\begin{align*}\frac{M_1}{D^3}\mathcal G(H,\mathcal M,f_1)&\ll \frac{M_1}{D^3}\left(x^{2\nu}R\left(Dx^{5\nu-1}M_2^2M_3\right)^\frac 34\right)\left(M_3^\frac 78R^{-\frac 18}\right)x^\eta\\
&\ll R^{\frac 78}M_1M_2^\frac 32M_3^\frac{13}8x^{6\nu-\frac 34}\ll x^{\frac 14+6\nu}M_2^\frac 12M_3^\frac{5}{8}R^{\frac 78}\ll x^{\frac 14+6\nu}M_2^\frac 98R^{\frac 78}.
\end{align*}
Since $M_2\leq \frac{R^\frac 52}{x^{1-24\nu}}$ and $R\leq Dx^{\frac{30}{59}-12\nu}$, the above then yields
\begin{gather}\label{GShp}\frac{M_1}{D^3}\mathcal G(H,\mathcal M,f_1)\ll R^{\frac{59}{16}}x^{-\frac 78+33\nu}\leq D^4x^{1-11\nu}\ll x^{1-10\nu},\end{gather}
which is as required.
\end{proof}
This completes the proof of Theorem \ref{Ggoal} for fixed $q$.  The result of this theorem can then be plugged into the results of Lemmas \ref{Gchi} and \ref{G1} to find that
\begin{align*}
\left|\mathcal T(\mathcal M)\right|\ll \frac{x^{1-4\nu}}{Dq}+D^\frac 32q^\frac 12 x^\nu.
\end{align*}
Since we require $Dq<x^{\frac{30}{59}-12\nu}$, we can see that the latter term will be inconsequential.  This completes the proof of Theorem \ref{mainishgoal}.

\section{The Shparlinski Method: Sum for a Range of $q$}\label{ShpQ}
For the sum over an average, we use Theorem \ref{Shp:T2}.  In particular, we take $\kappa=3\nu$ and replace $R^{o(1)}$ with $x^\eta$ for the bound in that theorem, giving us the following.
\begin{ThSh2}
Let $\kappa>0$ be a fixed real number, and let $Q\geq 1$ be sufficiently large.  For all but at most $(DQ)^{1-12\nu+o(1)}$ values of $R\in [DQ,2DQ]$,
\begin{gather*}\max_{(a,R)=1}|K_\chi(a)|\ll \left(HM_2+(HM_2)^\frac 34R^\frac 14\right)\left(M_3R^{-\frac 14}+M_3^\frac 12\right)R^{3\nu}x^{\eta}.
\end{gather*}
\end{ThSh2}
Let $\mathcal A$ denote the $q\in [Q,2Q)$ for which the bound for $K_\chi$ above holds, and let $\mathcal Z$ denote the remaining $q$.  If we take $q\geq \sqrt x$, we see that the number of elements in $\mathcal Z$ is bounded by
$$|\mathcal Z|\ll \frac{DQ^{1+o(1)}}{(DQ)^{12\nu}}\leq \frac{DQ^{1+o(1)}}{x^{6\nu}}.$$
So the size of $\mathcal Z$ is very small relative to $\mathcal A$.  This will allow us to prove the following.
\begin{theorem}\label{onavg}
If $\sqrt x< Q\leq x^{\frac{16}{31}-12\nu}$ for some $\varepsilon>0$ then $$\sum_{\stackrel{q\sim Q}{q\in \mathcal A}}\frac{M_1}{D^3}\left|\mathcal G(H,\mathcal M,f_1)\right|\ll Qx^{1-8\nu}.$$
Consequently, if $Q$ is as above then
\begin{align*}
\sum_{q\sim Q}\left|\mathcal T_{Dq}(\mathcal M)-\frac{1}{\phi(Dq)}\mathcal T_{Dq}^*(\mathcal M)\right|\ll \frac{x^{1-4\nu}}{D}.
\end{align*}
\end{theorem}

\begin{proof}
We begin with $\mathcal A$.  As in (\ref{HM2}), we have that $Dx^{2\nu}(HM_2)^\frac 34R^\frac 14\geq HM_2$, which means that for any $q\in \mathcal A$,
\begin{gather}\label{M3R}\frac{M_1}{D^3}\mathcal G(H,\mathcal M,f_1)\ll \frac{M_1R^{3\nu}x^\eta}{D^3}\left(Dx^{2\nu}(HM_2)^\frac 34R^\frac 14\right)\left(M_3R^{-\frac 14}+M_3^\frac 12\right).\end{gather}
Note that if $M_3\geq \sqrt R$ then $M_3R^{-\frac 14}\geq M_3^\frac 12$, while if $M_3<\sqrt R$ then $M_3R^{-\frac 14}<M_3^\frac 12$.  So we consider each of these two cases.

First, if $M_3\geq \sqrt R$ then
$$\frac{M_1}{D^3}\mathcal G(H,\mathcal M,f_1)\ll D^{-2}M_1x^{2\nu}R^{3\nu}x^{\eta}(HM_2)^\frac 34M_3\ll D^{-2}x^{6\nu}R^{3\nu}x^{\eta}\frac{M_1M_2^\frac 32M_3^\frac 74R^{\frac 34}}{x^\frac 34}\ll D^{-2}x^{\frac 14+8\nu}M_2^\frac 12M_3^\frac 34R^{\frac 34}.$$
Using our bound of  $M_2\leq \frac{R^{\frac 52}}{x^{1-24\nu}}$ and the fact that $M_3\leq M_2$,
$$\frac{M_1}{D^3}\mathcal G(H,\mathcal M,f_1)\ll D^{-2}x^{\frac 14+8\nu}M_2^\frac 54R^{\frac 34}\ll D^{-2}x^{-1+38\nu}R^{\frac{31}{8}}.$$
Taking $R\leq Dx^{\frac{16}{31}-12\nu}$, we can see that the above will be
$$\ll D^2x^{1-\frac{17}{2}\nu}\ll x^{1-8\nu}.$$
Next, we return to (\ref{M3R}) and let $M_3<\sqrt R$:
$$\frac{M_1}{D^3}\mathcal G(H,\mathcal M,f_1)\ll \frac{M_1}{D^2}x^{2\nu}R^{3\nu}(HM_2)^\frac 34R^{\frac 14}M_3^\frac 12\ll x^{8\nu}\frac{M_1M_2^\frac 32M_3^\frac 54}{x^\frac 34}\ll x^{\frac 14+8\nu}RM_2^\frac 12M_3^\frac 14.$$
Again using the bound for $M_2$, and using the fact that $M_3<\sqrt R$, the above is
$$\ll x^{\frac 14+8\nu}\left(\frac{R^{\frac 54}}{x^{\frac 12-12\nu}}\right)R^\frac 18\ll x^{-\frac 14+20\nu}R^{\frac{11}{8}}.$$
Taking $R<x^\frac{16}{31}$ gives that the above is
$$\ll x^{\frac{22}{31}-\frac 14+20\nu}.$$
The sum with $M_3\geq \sqrt R$ is larger, and hence\begin{align*}
\sum_{\substack{q\sim Q \\ q\in \mathcal A}}\frac{M_1}{D^3}\left|\mathcal G(H,\mathcal M,f_1)\right|\ll \sum_{\substack{q\sim Q \\ q\in \mathcal A}}x^{1-8\nu}\ll Qx^{1-8\nu}.
\end{align*}
Hence by Lemmas \ref{Gchi} and \ref{G1},
$$\sum_{\substack{q\sim Q \\ q\in \mathcal A}}\max_{(a,Dq)=1}\left|\mathcal T_{Dq}(\mathcal M)\right|\ll \frac{x^{1-4\nu}}{D}.$$
Finally, if $q\in \mathcal Z$ then we can bound $\mathcal T_{Dq}$ by Shiu's theorem:
$$\left|\mathop{\sum\sum\sum}_{\substack{(t_1,t_2,t_3)\in \mathcal M\\ t_1t_2t_3\equiv a\pmod{Dq}}}f_1(t_1)f_2(t_2)f_3(t_3)\right|\leq \sum_{\substack{n\leq x \\n\equiv a\pmod{Dq}}}\tau_3(n)\ll \frac xR\log^2 x.$$
So
$$\sum_{q\in \mathcal Z}\max_{(a,Dq)=1}\left|\mathcal T_{Dq}(\mathcal M)\right|\ll \frac{xR^{o(1)}\log^2 x}{x^{6\nu}}\ll \frac{x^{1-5\nu}}{D}.$$
Putting $\mathcal A$ and $\mathcal Z$ together,
$$\sum_{q\sim Q}\max_{(a,Dq)=1}\left|\mathcal T_{Dq}(\mathcal M)\right|\ll \frac{x^{1-4\nu}}{D}.$$
This proves the theorem, and it also completes the proof of Theorem \ref{mainishgoal2}.
\end{proof}

\section{The Sum $T_1$}\label{T1sec}

Now, we can finally bound the sum $T_1$ and complete the proof of Theorem \ref{Tsforall}.  We first prove Theorem \ref{terdivsum} with the following theorem.

\begin{theorem} Let $\sqrt x<q<x^{\frac{30}{59}-12\nu}$, and let $(a,Dq)=1$.  Then
$$\sum_{\substack{dm_1\leq x\\ d>D^*\\dm_1\equiv a\pmod{Dq}}}\lambda(d)=\frac{1}{\phi(q)}\sum_{\substack{dm_1\leq x \\ d>D^*\\ (dm_1,Dq)=1}}\lambda(d)+O\left(\frac{xL(1,\chi)}{Dq\log^8 x}+\frac{x}{Dq}L(1,\chi)^2\log^2 x\right).$$
Additionally, let $Q<x^{\frac{16}{31}-12\nu}$.  Then
$$\sum_{q\sim Q}\left|\sum_{\substack{dm_1\leq x \\ d>D^* \\dm_1\equiv a\pmod{Dq}}}\lambda(d)-\frac{1}{\phi(q)}\sum_{\substack{dm_1\leq x\\ d>D^*\\ (dm_1,Dq)=1}}\lambda(d)\right|\ll \frac{xL(1,\chi)}{D\log^8 x}+\frac{x}{D}L(1,\chi)^2\log^2 x.$$
\end{theorem}
\begin{proof}
From Lemmas \ref{easy1}, \ref{easy2}, and \ref{binarylemma} and Theorem \ref{mainishgoal}, we have shown that for fixed $q$,
\begin{align*}
\mathop{\sum\sum}_{\substack{dm_1\leq x \\ d>D^* \\ dm_1\equiv a\pmod{Dq}}}\lambda(d)=&\frac{1}{\phi(Dq)}\mathop{\sum\sum}_{\substack{dm_1\leq x \\ d>D^* \\ (dm_1,Dq)=1}}\lambda(d)+O(E_1)\end{align*}
where
\begin{align*}
E_1=&\frac{x}{Dq^{1+\alpha}}+\frac{x}{Dq}L(1,\chi)^2\log^2x+\frac{x^{1-\frac \rho 2}}{Dq}+\frac{xL(1,\chi)}{Dq\log^8 x}.\end{align*}
Clearly, the largest term in $E_1$ is either the second or fourth term, depending how small $L(1,\chi)$ is.  This proves the first half of the theorem. 

For the second half, we use Theorem \ref{mainishgoal2} instead of Theorem \ref{mainishgoal}.  The bounds, however, are largely the same, as
\begin{align*}
\sum_{q\sim Q}\left|\mathop{\sum\sum}_{\substack{dm_1\leq x \\ d>D^* \\ dm_1\equiv a\pmod{Dq}}}\lambda(d)-\sum_{q\sim Q}\frac{1}{\phi(Dq)}\mathop{\sum\sum}_{\substack{dm_1\leq x \\ d>D^* \\ (dm_1,Dq)=1}}\lambda(d)\right|\ll \sum_{q\sim Q}E_1,\end{align*}
where $E_1$ is the same as above.  Summing over $q\in Q$ then gives the second half of the theorem.
\end{proof}

Using the above, we can bound $T_1$.
\begin{theorem}\label{T1final}    Let $0<\nu<\frac{1}{10,000}$ with $x^{\frac{\nu}{5}}>x^\eta>D$.  If $q\leq \frac{x^{\frac{30}{59}-12\nu}}{D^2}$ then
$$T_1\ll \frac{x}{\phi(q)}L(1,\chi)\log^2 x\log\log x+\frac{x}{q}L(1,\chi)^2\log^3x.$$
If $q\leq \frac{x^{\frac{16}{31}-12\nu}}{D^2}$ then
$$\sum_{q\sim Q}\max_{(a,D)=1}|T_1(a,q)|\ll xL(1,\chi)\log^2 x\log \log x+xL(1,\chi)^2\log^3x.$$
\end{theorem}
\begin{proof}
We begin with the case of fixed $q$.  As usual, we sum over the appropriate classes mod $Dq$ to construct the sum for $a$ mod $q$.  Noting that $m_2\leq D^*=D^{2+\alpha}$, our assumption on $q$ then implies that $q\leq \left(\frac{x}{m_2}\right)^{\frac{30}{59}-12\nu}$.  Hence,
\begin{align*}
\sum_{m_2\leq D^*}&\mathop{\sum\sum}_{\substack{dm_1\leq \frac{x}{m_2} \\ d>D^* \\ dm_1\equiv a\overline{m_2}\pmod{q}}}\lambda(d)\\
=&\sum_{m_2\leq D^*}\sum_{\substack{0\leq a'\leq D-1\\(D,a'q+a)=1}}\mathop{\sum\sum}_{\substack{dm_1\leq \frac{x}{m_2} \\ d>D^* \\ dm_1\equiv a\overline{m_2}\pmod{Dq}}}\lambda(d)\\
=&\sum_{m_2\leq D^*}\sum_{\substack{0\leq a'\leq D-1\\(D,a'q+a)=1}}\left(\frac{1}{\phi(Dq)}\mathop{\sum\sum}_{\substack{dm_1\leq \frac{x}{m_2}\\ d>D^*  \\ (dm_1m_2,Dq)=1}}\lambda(d)+O\left(\frac{xL(1,\chi)}{m_2Dq\log^8 x}+\frac{x}{m_2Dq}L(1,\chi)^2\log^2 x\right)\right)\\
=&\frac{1}{\phi(Dq)}\sum_{m_2\leq D^*}\sum_{\substack{0\leq a'\leq D-1\\(D,a'q+a)=1}}\mathop{\sum\sum}_{\substack{dm_1\leq \frac{x}{m_2} \\ d>D^* \\ (dm_1m_2,Dq)=1}}\lambda(d)+O\left(\frac{xL(1,\chi)}{q\log^7 x}+\frac{x}{q}L(1,\chi)^2\log^3 x\right),
\end{align*}
where the third line comes from the previous lemma.

For the main term, we have
\begin{align*}
&\frac{1}{\phi(Dq)}\sum_{\substack{0\leq a'\leq D-1\\(D,a'q+a)=1}}\sum_{m_2\leq D^*}\mathop{\sum\sum}_{\substack{dm_1\leq \frac{x}{m_2} \\ d>D^* \\ (dm_1m_2,q)=1}}\lambda(d)\\
&\ll \frac{1}{\phi(Dq)}\sum_{\substack{0\leq a'\leq D-1\\(D,a'q+a)=1}}\sum_{m_2\leq D^*}\left(\sum_{D^*<d\leq \sqrt{\frac{x}{m_2}}} \lambda(d)\sum_{m_1\leq \frac{x}{m_2d}}1+\sum_{m_1\leq \sqrt{\frac{x}{m_2}}}\sum_{d\leq \frac{x}{m_1m_2}}\lambda(d)\right)\\
&\ll \frac{1}{\phi(Dq)}\sum_{\substack{0\leq a'\leq D-1\\(D,a'q+a)=1}}\sum_{m_2\leq D^*}\left(\sum_{D^*<d\leq \sqrt{\frac{x}{m_2}}} \frac{x\lambda(d)}{m_2d}+\sum_{m_1\leq \sqrt{\frac{x}{m_2}}}\frac{xL(1,\chi)}{m_1m_2}\right)\\
&\ll \frac{xD}{\phi(Dq)}L(1,\chi)\log^2 x\ll \frac{xD}{\phi(D)\phi(q)}L(1,\chi)\log^2 x\ll \frac{x}{\phi(q)}L(1,\chi)\log^2 x\log\log D.
\end{align*}
This completes the computation of $T_1$ for fixed $q$.

For the sum over $q\sim Q$, the error terms are handled nearly identically.  For the main term, we pick up where we left off above:
\begin{align*}
\sum_{q\sim Q}\frac{1}{\phi(Dq)}&\sum_{\substack{0\leq a'\leq D-1\\(D,a'q+a)=1}}\sum_{m_2\leq D^*}\mathop{\sum\sum}_{\substack{dm_1\leq \frac{x}{m_2} \\ d>D^* \\ (dm_1m_2,q)=1}}\lambda(d)\\
\ll &\sum_{q\sim Q}\frac{x}{\phi(q)}L(1,\chi)\log^2 x\log\log D\ll xL(1,\chi)\log^2 x\log\log D.
\end{align*}
This completes the proof of the theorem.
\end{proof}
\section{Acknowledgements}
I would like to thank Stylianos Sachpazis for pointing out an error in an earlier draft.  I would also like to thank Igor Shparlinski for sending some helpful results about Kloosterman sums.  Thanks also to the anonymous referees for many helpful suggestions and corrections.

\bibliographystyle{line}

\begin{thebibliography}{[HD]}
\normalsize
\bibitem[BC]{BC} S. Bettin and V. Chandee, Trilinear forms with Kloosterman fractions,
Adv. Math., 328 (2018), 1234-1262.
\bibitem[Di]{Di} K. Dickman, On the frequency of numbers containing prime factors
of a certain relative magnitude, Ark. Mat. Astr. Fys. 22 (1930), 1-14.
\bibitem[FI85]{FI85} J. B. Friedlander and H. Iwaniec, Incomplete Kloosterman sums and a divisor problem,
Ann. of Math. (2) 121 (1985), no. 2, 319-350.
\bibitem[FI03]{FI03} J. B. Friedlander and H. Iwaniec, Exceptional characters and prime numbers in arithmetic progressions, Int. Math. Res. Not. 37 (2003), 2033-2050.
\bibitem[FI18]{FI18} J. B. Friedlander and H. Iwaniec. A note on Dirichlet L-functions. Expo. Math., 36(3-4):343-
350, 2018.
\bibitem[FS]{FS} \'{E}. Fouvry and I. E. Shparlinski, On a ternary quadratic form over primes, Acta Arith.,
150(3) (2011), 285-314.
\bibitem[Gr]{Gr} T.H. Gronwall, Sur les s\'{e}ries de Dirichlet correspondant \`{a} des charact\`{e}res complexes, Rendiconti di Palermo, 35 (1913), 145-159.
\bibitem[HB79]{HB79} D. R. Heath-Brown, The fourth power moment of the Riemann Zeta-function, Proc.
London Math. Soc. (3) 38 (1979), 385-422.
\bibitem[HB83]{HB83} D. R. Heath-Brown, Prime twins and Siegel zeros, Proc. London Math. Soc. (3) 47 (1983), no. 2, 193-224.
\bibitem[HB86]{HB86} D.R. Heath–Brown: The divisor function $d_3(n)$ in arithmetic progressions, Acta Arith. 47 (1986), 29-56
\bibitem[Ho]{Ho} C. Hooley, An asymptotic formula in the theory of numbers, Proc. London Math. Soc.
(3) 7 (1957), 396-413.
\bibitem[Ir]{Ir} A.J. Irving, Average bounds for Kloosterman sums over primes, Funct. Approx. Comment. Math. 2 (51) (2014), 221-235.
\bibitem[FKM]{FKM} \'{E}. Fouvry, E. Kowalski, and P. Michel, On the exponent of distribution of the ternary divisor function,
Mathematika 61 (1) (2015), 121-144.
\bibitem[FR]{FR} \'{E}. Fouvry and M. Radziwi{\l}{\l}. Level of distribution of unbalanced convolutions. Ann. Sci. Ec. Norm. Super.
(4), 55(2) (2022), 537-568.
\bibitem[La1]{La} E. Landau, Über die Klassenzahl imaginär-quadratischer Zahlkörper, Nachr. Ges. Wiss. Göttingen (1918), 285-295.
\bibitem[La2]{Land} B. Landreau, Majorations de fonctions arithmetiques en moyenne sur des ensembles de faible
densit\'{e},  S\'{e}min. Th´eor. Nombres Bordeaux 16 (1987–1988), 1-18.
\bibitem[MM]{MaMe} K. Matom\"{a}ki and J. Merikoski. Siegel zeros, twin primes, Goldbach’s conjecture, and primes in short
intervals, Int. Math. Res. Not. 23 (2023), 20337–20384.
\bibitem[Ma]{MaBT} J. Maynard, On the Brun-Titchmarsh theorem, Acta Arith., 157(3) (2013), 249-296.
\bibitem[Se]{Se} A. Selberg, Lectures on Sieves, Collected Papers, vol. II, Springer-Verlag, Berlin, 1991
\bibitem[Sh24]{Sha} P. Sharma, Bilinear sums with GL(2) coefficients and the exponent of distribution of $d_3$, https://arxiv.org/abs/2303.06087.
\bibitem[Sh80]{Sh} P. Shiu, A Brun-Titchmarsh theorem for multiplicative functions, J. Reine Angew. Math. 313
(1980), 161-170.
\bibitem[Sh19]{Shp} I. Shparlinski, Bounds on average values of double incomplete Kloosterman sums, J. Number Theory, 203 (2019), 1-11.
\bibitem[Si]{Si} C. L. Siegel, \:{U}ber die Classenzahl quadratischer Zahlk\:{o}rper, Acta Arith. 1 (1935), 83–86.
\bibitem[Ti]{Ti} E.C. Titchmarsh, A divisor problem, Rendiconti Palermo, 54 (1930), 414-429.
\bibitem[TT]{TT} T. Tao and J. Ter\"{a}v\"{a}inen, The Hardy-Littlewood-Chowla conjecture in the presence of a Siegel zero, Jour. London Math. Soc. 106 (4) (2022), 3317-3378.
\bibitem[Wr]{WrS} T. Wright, Prime tuples and Siegel zeroes, Bull. London Math. Soc, https://doi.org/10.1112/blms.12956 (to appear).
\bibitem[Zh]{ZhS} Y. Zhang, Discrete mean estimates and the Landau-Siegel zero, https://arxiv.org/abs/2211.02515.
\end{thebibliography}

\end{document}